\documentclass[
,reqno]{amsart}
\usepackage{latexsym,amsmath,amssymb,amscd}
\usepackage[all]{xy}
\usepackage{mathrsfs}

\def\today{\ifcase \month \or
   January \or February \or March \or April \or
   May \or June \or July \or August \or
   September \or October \or November \or December \fi
   \space\number\day , \number\year}
   
\makeatletter
  \newcommand\@dotsep{4.5}
  \def\@tocline#1#2#3#4#5#6#7{\relax
     \ifnum #1>\c@tocdepth 
     \else
     \par \addpenalty\@secpenalty\addvspace{#2}%
     \begingroup \hyphenpenalty\@M
     \@ifempty{#4}{%
     \@tempdima\csname r@tocindent\number#1\endcsname\relax
        }{%
         \@tempdima#4\relax
           }%
      \parindent\z@ \leftskip#3\relax \advance\leftskip\@tempdima\relax
      \rightskip\@pnumwidth plus1em \parfillskip-\@pnumwidth
       #5\leavevmode\hskip-\@tempdima #6\relax
       \leaders\hbox{$\m@th
       \mkern \@dotsep mu\hbox{.}\mkern \@dotsep mu$}\hfill
       \hbox to\@pnumwidth{\@tocpagenum{#7}}\par
       \nobreak
        \endgroup
         \fi}
\makeatother 

\begin{document}


\makeatletter
\@addtoreset{figure}{section}
\def\thefigure{\thesection.\@arabic\c@figure}
\def\fps@figure{h,t}
\@addtoreset{table}{bsection}

\def\thetable{\thesection.\@arabic\c@table}
\def\fps@table{h, t}
\@addtoreset{equation}{section}
\def\theequation{
\arabic{equation}}
\makeatother

\newcommand{\bfi}{\bfseries\itshape}

\newtheorem{theorem}{Theorem}
\newtheorem{acknowledgment}[theorem]{Acknowledgment}
\newtheorem{corollary}[theorem]{Corollary}
\newtheorem{definition}[theorem]{Definition}
\newtheorem{example}[theorem]{Example}
\newtheorem{lemma}[theorem]{Lemma}
\newtheorem{notation}[theorem]{Notation}
\newtheorem{proposition}[theorem]{Proposition}
\newtheorem{remark}[theorem]{Remark}
\newtheorem{setting}[theorem]{Setting}
\newtheorem{hypothesis}[theorem]{Standing Hypothesis}

\numberwithin{theorem}{section}
\numberwithin{equation}{section}


\def\C{\mathbb{C}} 
\def\N{\mathbb{N}} 
\def\R{\mathbb{R}} 
\def\T{\mathbb{T}} 
\def\Z{\mathbb{Z}}
\def\A{{\mathcal A}} 
\def\F{\mathcal F} 
\def\H{\mathcal H} 
\def\K{\mathcal K}
\def\G{\mathcal{G}}
\def\Int{\mathfrak{Int}}
\def\la{\langle} 
\def\ra{\rangle} 
\def\Op{\mathfrak{Op}} 
\def\X{\mathcal X}
\def\GG{\mathfrak{G}}
\def\p{\parallel} 
\def\si{\sigma}
\def\Si{\Sigma}
\def\td{\star}
\def\({\left(}
\def\){\right)}
\def\[{\left[}
\def\]{\right]}
\def\<{\left\langle}
\def\>{\right\rangle}


\renewcommand{\1}{{\bf 1}}
\newcommand{\Ad}{{\rm Ad}}
\newcommand{\Alg}{{\rm Alg}\,}
\newcommand{\alg}{*{\rm alg}}
\newcommand{\Aut}{{\rm Aut}\,}
\newcommand{\ad}{{\rm ad}}
\newcommand{\Borel}{{\rm Borel}}
\newcommand{\botimes}{\bar\otimes}
\newcommand{\Ci}{{\mathcal C}^\infty}
\newcommand{\Cint}{{\mathcal C}^\infty_{\rm int}}
\newcommand{\Cpol}{{\mathcal C}^\infty_{\rm pol}}
\newcommand{\Der}{{\rm Der}\,}
\newcommand{\de}{{\rm d}}
\newcommand{\ee}{{\rm e}}
\newcommand{\End}{{\rm End}\,}
\newcommand{\ev}{{\rm ev}}
\newcommand{\id}{{\rm id}}
\newcommand{\ie}{{\rm i}}
\newcommand{\ind}{\mathop{\rm ind}}
\newcommand{\GL}{{\rm GL}}
\newcommand{\gl}{{{\mathfrak g}{\mathfrak l}}}
\newcommand{\Hom}{{\rm Hom}\,}
\newcommand{\Img}{{\rm Im}\,}
\newcommand{\Ind}{{\rm Ind}}
\newcommand{\Ker}{{\rm Ker}\,}
\renewcommand{\L}{L}
\newcommand{\Lie}{\text{\bf L}}
\newcommand{\Mt}{{{\mathcal M}_{\text t}}}
\newcommand{\m}{\text{\bf m}}
\newcommand{\pr}{{\rm pr}}
\newcommand{\Ran}{{\rm Ran}\,}
\renewcommand{\Re}{{\rm Re}\,}
\newcommand{\spa}{{\rm span}\,}
\newcommand{\Tr}{{\rm Tr}\,}

\newcommand{\SQ}{\text{\rm SQ}}

\newcommand{\UC}{{{\mathcal U}{\mathcal C}}}
\newcommand{\UCb}{{{\mathcal U}{\mathcal C}_b}}
\newcommand{\LUCb}{{{\mathcal L}{\mathcal U}{\mathcal C}_b}}
\newcommand{\RUCb}{{{\mathcal R}{\mathcal U}{\mathcal C}_b}}
\newcommand{\LUCl}{{{\mathcal L}{\mathcal U}{\mathcal C}}_{\rm loc}}
\newcommand{\RUCl}{{{\mathcal R}{\mathcal U}{\mathcal C}}_{\rm loc}}
\newcommand{\UCl}{{{\mathcal U}{\mathcal C}}_{\rm loc}}

\newcommand{\CC}{{\mathbb C}}
\newcommand{\RR}{{\mathbb R}}
\newcommand{\TT}{{\mathbb T}}

\newcommand{\Ac}{{\mathcal A}}
\newcommand{\Bc}{{\mathbb B}}
\newcommand{\Cc}{{\mathcal C}}
\newcommand{\Dc}{{\mathcal D}}
\newcommand{\Ec}{{\mathcal E}}
\newcommand{\Fc}{{\mathcal F}}
\newcommand{\Hc}{{\mathcal H}}
\newcommand{\Jc}{{\mathcal J}}
\newcommand{\Kc}{{\mathcal K}}
\newcommand{\Lc}{{\mathcal L}}
\renewcommand{\Mc}{{\mathcal M}}
\newcommand{\Nc}{{\mathcal N}}
\newcommand{\Oc}{{\mathcal O}}
\newcommand{\Pc}{{\mathcal P}}
\newcommand{\Qc}{{\mathcal Q}}
\newcommand{\Rac}{{\mathcal R}}
\newcommand{\Sc}{{\mathcal S}}
\newcommand{\Tc}{{\mathcal T}}
\newcommand{\Uc}{{\mathcal U}}
\newcommand{\Vc}{{\mathcal V}}
\newcommand{\Wig}{{\mathcal W}}
\newcommand{\Xc}{{\mathcal X}}
\newcommand{\Yc}{{\mathcal Y}}

\newcommand{\Bg}{{\mathfrak B}}
\newcommand{\Dg}{{\mathfrak D}}
\newcommand{\Fg}{{\mathfrak F}}
\newcommand{\Gg}{{\mathfrak G}}
\newcommand{\Ig}{{\mathfrak I}}
\newcommand{\Jg}{{\mathfrak J}}
\newcommand{\Lg}{{\mathfrak L}}
\newcommand{\Pg}{{\mathfrak P}}
\newcommand{\Rg}{{\mathfrak R}}
\newcommand{\Sg}{{\mathfrak S}}
\newcommand{\Xg}{{\mathfrak X}}
\newcommand{\Yg}{{\mathfrak Y}}
\newcommand{\Zg}{{\mathfrak Z}}

\newcommand{\ag}{{\mathfrak a}}
\newcommand{\bg}{{\mathfrak b}}
\newcommand{\dg}{{\mathfrak d}}
\renewcommand{\gg}{{\mathfrak g}}
\newcommand{\hg}{{\mathfrak h}}
\newcommand{\kg}{{\mathfrak k}}
\newcommand{\mg}{{\mathfrak m}}
\newcommand{\n}{{\mathfrak n}}
\newcommand{\og}{{\mathfrak o}}
\newcommand{\pg}{{\mathfrak p}}
\newcommand{\sg}{{\mathfrak s}}
\newcommand{\tg}{{\mathfrak t}}
\newcommand{\ug}{{\mathfrak u}}
\newcommand{\zg}{{\mathfrak z}}

\newcommand{\ZZ}{\mathbb Z}
\newcommand{\NN}{\mathbb N}
\newcommand{\BB}{\mathbb B}
\newcommand{\HH}{\mathbb H}

\newcommand{\ep}{\varepsilon}

\makeatletter
\title[Symbol calculus of square-integrable operator-valued maps]{Symbol calculus of 
square-integrable operator-valued maps}
\author{Ingrid Belti\c t\u a,  
Daniel Belti\c t\u a 
and Marius M\u antoiu}
\address{Institute of Mathematics ``Simion Stoilow'' 
of the Romanian Academy, 
P.O. Box 1-764, Bucharest, Romania}
\email{ingrid.beltita@gmail.com, Ingrid.Beltita@imar.ro}
\address{Institute of Mathematics ``Simion Stoilow'' 
of the Romanian Academy, 
P.O. Box 1-764, Bucharest, Romania}
\email{beltita@gmail.com, Daniel.Beltita@imar.ro}
\address{Departamento de Matem\'aticas, Universidad de Chile, Las Palmeras 3425,
Casilla 653, Santiago, Chile}
\email{mantoiu@uchile.cl}
\keywords{symbol calculus, Hilbert algebra, Berezin-Toeplitz operator, square-integrable representation} 
\subjclass[2000]{Primary 46L65; Secondary 22E66, 35S05, 46K15, 46H30}
\thanks{The research of I. Belti\c t\u a and D. Belti\c t\u a has been partially supported by the Grant
of the Romanian National Authority for Scientific Research, CNCS-UEFISCDI,
project number PN-II-ID-PCE-2011-3-0131. 
M. M\u antoiu has been supported by the Fondecyt Project 1120300 and the N\'ucleo Milenio de F\'isica Matem\'atica RC120002.}
\makeatother

\begin{abstract} 
We develop an abstract framework for the investigation of quantization and dequantization procedures 
based on orthogonality relations that do not necessarily involve group representations. 
To illustrate the usefulness of our abstract method we show that it behaves well 
with respect to the infinite tensor products. 
This construction subsumes examples coming from the study of magnetic Weyl calculus, 
the magnetic pseudo-differential Weyl calculus,   
the metaplectic representation on locally compact abelian groups,   
irreducible representations associated with finite-dimensional coadjoint orbits of some 
special infinite-dimensional Lie groups,  
and the square-integrability properties shared by arbitrary irreducible representations of nilpotent Lie groups. 
\end{abstract}

\maketitle

\tableofcontents

\section{Introduction}\label{intro}

Square-integrable representations of locally compact groups  
play a well-known role in Lie theory, representation theory, and their applications to physics. 
The present paper is devoted to developing a set of techniques applicable 
to operator valued maps on measure spaces $\pi\colon(\Sigma,\mu)\to\BB(\Hc)$ 
that satisfy a square integrability property analogous to that of locally compact group representations 
(see \eqref{porc} below) 
although $\pi$ may not be a group representation and $\mu$ may not be a Haar measure. 
This investigation was motivated by several situations when actually 
$\Sigma$ is a group that fails to be locally compact so it does not admit any Haar measure 
(as for instance in the study of the canonical commutation relations 
where one has been looking for suitable substitutes of the group algebra 
for inductive limit groups \cite{Gru97,GN09,GN13}), 
or when $\Sigma$ is locally compact but $\pi$ is not even a projective group representation 
(see for instance the orthogonality relations for irreducible representations 
of nilpotent Lie groups: \cite{HM79,Pe,Wo13}, and the references therein; 
or the magnetic Weyl calculus \cite{MP1,IMP, MP5,BB1,BB11a}).

From a more technical point of view, 
this article, as many others, is concerned with symbol calculus, under certain circumstances also called {\it quantization},
seen as a systematic way to associate operators in some
infinite-dimensional vector space with functions almost everywhere defined in a suitable related set $\Sigma$ 
endowed with a measure. 
We have in mind mainly operators acting in a Hilbert space $\H$,
but other type of topological vector spaces will also be considered. 
Moreover, in order to define symbols for arbitrary bounded linear operators on $\H$, 
we will also need (in Subsection~\ref{platica}) 
extensions of the symbol calculus beyond spaces of functions on $\Sigma$, 
in the same way as the classical Weyl calculus on $\RR^n$ needs to be extended from 
functions to tempered distributions on $\RR^{2n}$.

Among the different strategies to start and 
motivate 
a symbol calculus, there are two, dual to each other.
The first one, inspired by Weyl's quantization procedure, consists in associating to each point $s$ of the space $\Sigma$
a bounded linear operator $\pi(s)$ in $\H$. 
This mapping
$s\mapsto\pi(s)$, while not supposed to be unitarily valued or to possess group-like properties,
would benefit from some regularity requirements.
Boundedness and weak continuity are good properties, 
and yet a square integrability condition with respect to some measure
$\mu$ on $\Sigma$ (generalizing the notion of square integrable representation of a group)
is the best starting point. 
Then operators $\Pi(f)$ are associated to suitable functions $f$ on
$\Si$ by integration techniques (cf. (\ref{guanaco})), and square integrability plays an important role in identifying Hilbert-Schmidt operators as corresponding by quantization to $L^2$-functions. 
Simple examples show that not all the elements in $L^2(\Si,\mu)$ need to be involved.

The dual approach is to give a priori the symbols (functions defined on $\Sigma$) of all the rank-one operators.
Then the symbols of more general operators are obtained by superposition, modelled by integration on $\Sigma$,
followed eventually by extension techniques. 
If suitably implemented, the construction is essentially the inverse of the one described above.
But this is achieved only after the formalism has been extended; the operators $\pi(s)$ are almost never Hilbert-Schmidt.

Many classes of operators form $^*$-algebras under operator multiplication and taking adjoints.
Clearly, it is desirable to use the quantization
to induce isomorphic versions of these $^*$-algebras on classes of symbols. 
As a matter of fact, due to the square integrability
assumption, one obtains compatible scalar products on the $^*$-algebras making them Hilbert algebras. 
This already makes available
extension techniques permitting the treatment of symbols not associated to Hilbert-Schmidt operators. 
But most of the known examples
strongly suggest the existence of an extra mathematical structure, resulting in Gelfand triples both at the level of vectors
and at the level of symbols, suitably interconnected. 
One could simply recall the role played in pseudodifferential theory
by the Schwartz space and its dual, the space of tempered distributions. 
Another example, leading to Gelfand triples of Banach spaces,
is the Segal algebra, available on locally compact groups. 
In our general framework we will indicate
a systematic way to construct Gelfand triples connected to 
the symbol calculus associated to the family $\{\pi(s)\mid s\in\Si\}$;
that will turn out to have a rich algebraic content. 

The object of Sections~\ref{vaca} and \ref{vacuta} is the construction of the symbol calculus associated to 
the data $\(\Si,\mu,\pi,\H\)$, 
where $\Sigma$ is a space endowed with a measure $\mu$. 
That space serves as a family of indices for
a set of bounded operators $\{\pi(s)\!\mid\!s\in\Si\,\}$ in the Hilbert space $\H$.
We do not assume that $\pi(s)$ is unitary and we do not require anything about the product $\pi(s)\pi(t)$ for $s,t\in\Si$. 
The map $\pi(\cdot)$ is assumed bounded and weakly continuous. 
The main requirement is relation (\ref{porc}), 
a condition of square integrability extending a well-known concept from group representation theory \cite{Di}.

In Section \ref{vaca} we show that the class of square-integrable families of operators 
is closed under some basic operations as the compressions and the tensor products. 
We also show that these families are irreducible in the sense that their commutant is trivial and 
they do not have any nontrivial common invariant subspace, 
which suggests the interesting problem of pointing out the topological groups for which every unitary irreducible representation admits a measure on the group for which the representation is square integrable; see for instance Corollaries \ref{irred2} and \ref{irred3} for 
answers to this question in the case of compact and nilpotent Lie groups, respectively. 

In Section \ref{vacuta}, the first purpose is to raise the family $\pi$, essentially by integration, to a correspondence $f\mapsto\Pi(f)$ sending a closed subspace of $\L^2(\Si;\mu)$ to the ideal of all Hilbert-Schmidt operators in $\H$. Actually the fact that $\Pi$ is
``an integrated form'' of $\pi$ (in the spirit of group representation theory) is only seen a posteriori; the initial construction is just
based on the the ``representation coefficient'' map $\Phi$. 
The linear maps $\Pi,\Phi$ and $\Lambda$ are isomorphisms of $H^*$-algebras. 
By transport of structure via these isomorphisms, 
one then defines classes of trace-class, compact and bounded-type symbols forming Banach
$^*$-algebras; Radon measures on $\Si$ can also be incorporated when $\Sigma$ is a locally compact space.  

We also develop the dequantization procedure  for the operator calculus. 
That is, we develop methods for recovering the symbol of a given operator. 
In order to do that in an effective way, we need to explore new spaces of symbols 
and their natural composition law that recovers the twisted convolution 
in the case of group representations and corresponds to the composition of operators 
in the representation space. 

Section~\ref{graur_sect} deals with the Gelfand triples that occur in our general framework, 
in connection with suitable dense subspaces of the Hilbert space under consideration. 
This is suggested by the classical example of the Schwartz space $\Sc(\RR^n)$ 
of rapidly decreasing functions 
on $\RR^n$, which is continuously and densely embedded into the Hilbert space $L^2(\RR^n)$ 
and is closely related to the square-integrable family of unitary operators 
$\RR^n\times\RR^n\to\BB(L^2(\RR^n))$ defined by the Weyl system. 
We study the abstract version of the operators $\Sc(\RR^n)\to\Sc'(\RR^n)$ 
and some related structures, which provide a unifying perspective 
on different types of applications in Section~\ref{marlita}. 

In Section~\ref{brabusca} we merely develop some very basic aspects of the Berezin-Toeplitz quantization 
in our abstract framework, in order to suggest how this important topic fits in our picture. 

Section~\ref{next} explores the infinite tensor products of square-integrable families of operators, 
a circle of ideas that plays an important role in the representation theory of canonical commutation relations (CCR)
with infinitely many degrees of freedom; see for instance \cite{KM65,Heg69,Heg70,Re70}. 
We prove that these infinite tensor products always have a certain property of approximate square integrability 
(Theorem~\ref{inf}) 
and then we also discuss the Berezin-Toeplitz quantization to a very limited extent, 
which already suggests however that several interesting problems arise in this area. 
It is noteworthy that the relationship between the CCR and the infinite tensor products was 
also studied studied in \cite{Gru97,GN09,GN13} from a different perspective. 

Finally, in Section~\ref{marlita} we briefly present four topics from the earlier literature 
where one can find special cases of the general ideas developed in the present paper:  
1. the magnetic pseudo-differential Weyl calculus;  
2. the study of the metaplectic representation on locally compact abelian groups;  
3. irreducible representations associated with finite-dimensional coadjoint orbits of some 
special infinite-dimensional Lie groups; 
4. the square-integrability properties shared by arbitrary irreducible representations of nilpotent Lie groups. 

It would also be quite interesting to understand the relationship between our abstract approach 
and the Weyl and Berezin calculus on bounded symmetric domains as developed for instance 
in \cite{UU94} and \cite{AU07}. 

\subsection*{Preliminary conventions and notations}
A convenient reference for square-integrable representations of locally compact groups 
and their role in representation theory is \cite[App. VII]{Ne00}; see also the references therein. 

If $\Si$ is a topological space (always  Hausdorff), we set $BC(\Si)$ for the $C^*$-algebra of all bounded continuous complex-valued functions on $\Si$. 
If $\Si$ is locally compact we write $C_0(\Si)$ for the $C^*$-algebra of the continuous functions vanishing at infinity. 
For any measure $\mu$ on $\Si$ and $q\in[1,\infty]$, we denote by $\L^q(\Si;\mu)$ the usual Lebesgue space of order $q$ on $(\Si,\mu)$.

For two complex Hausdorff locally convex spaces $\mathcal E$ and $\mathcal F$, 
we will write $\mathcal E\otimes \mathcal F$ for the algebraic tensor product. 
When endowed with the projective topology, that space will be denoted by $\mathcal E\otimes_p\!\mathcal F$ and
its completion in this topology by $\mathcal E\widehat\otimes_p \mathcal F$. 
Analogously, $\mathcal E\widehat\otimes_i \mathcal F$ will be the completion of
$\mathcal E\otimes_i\!\mathcal F$, which is $\mathcal E\otimes \mathcal F$ endowed with the injective topology.

Under the same assumptions on $\mathcal E$ and $\mathcal F$, we denote by $\mathbb B(\mathcal E,\mathcal F)$ the vector space
of all linear {\it continuous} operators from $\mathcal E$ into $\mathcal F$ and use the abbreviation $\mathbb B(\mathcal E):=\mathbb B(\mathcal E,\mathcal E)$.
Then $\mathcal E':=\mathbb B(\mathcal E,\C)$ is the (topological) dual of $\mathcal E$.

We recall from \cite{Di}  that a Hilbert algebra is a $^*$-algebra $(\mathscr A,\#,^\#)$ endowed with a scalar product
$\<\cdot,\cdot\>:\mathscr A\times\mathscr A\rightarrow\mathbb C$
such that
\begin{enumerate}
\item
one has $\,\<g^\#,f^\#\>=\<f,g\>$ for all $\,f,g\in\A$\,,
\item
one has $\,\<f\#g,h\>=\<g,f^\#\# h\>$ for all $\,f,g,h\in\A$\,,
\item
for every $g\in\mathscr A$, the map $\,{\sf L}_g:\mathscr A\rightarrow\mathscr A,\,{\sf L}_g(f):=g\#\!f$ is continuous.
\item
$\mathscr A\#\mathscr A$ is total in $\mathscr A$.
\end{enumerate}
A complete Hilbert algebra is called an \emph{$H^*$-algebra}.

Clearly one also has $\,\<f\#g,h\>=\<f,h\# g^\#\>\,$ for all $\,f,g,h\in\mathscr A$
and the map ${\sf R}_g:\mathscr A\rightarrow\mathscr A$ given by $\,{\sf R}_g(f):=f\#g$ is also continuous; therefore
$\mathscr A\times\mathscr A\overset{\#}{\rightarrow}\mathscr A$ is separately continuous.

To give some basic examples, let us fix a complex Hilbert space $\H$\,;
by convention the scalar product $\<\cdot,\cdot\>$ is anti-linear
in the second variable and we denote by $\overline{\H}$ the conjugate space of $\H$.
By Riesz Theorem, the dual $\H'$ of $\H$ is canonically antilinearly
isomorphic to $\H$, so there is a linear isomorphism permitting the identification of $\overline\H$ with $\H'$.
Recall that the space $\mathbb B_2(\H)$ of Hilbert-Schmidt operators on $\H$ forms a
$^*$-ideal in $\mathbb B(\H)$ and a Hilbert space with the scalar product $\<S,T\>_{\mathbb B_2(\H)}:=\Tr\(ST^*\)$.
Actually $\mathbb B_2(\H)$ is a $H^*$-algebra; the subspace $\mathbb B_2(\H)\mathbb B_2(\H)$ (coinciding with the
$^*$-ideal $\mathbb B_1(\H)$ of all trace-class operators) is dense in $\mathbb B_2(\H)$.

Let us denote by $\Lambda$ the canonical unitary operator 
\begin{equation}\label{Lambda_def}
\Lambda\colon\H\widehat\otimes\overline{\H}\to\mathbb B_2(\H),\quad \Lambda(u\otimes v):=\lambda_{u,v}:=
\la \cdot,v\ra u,
\end{equation}
where $\H\widehat\otimes\overline{\H}$ stands for the Hilbert completion of the algebraic tensor product $\H\otimes\overline\H$.
On the space $\H\otimes\overline\H$ we consider the unique structure of  $H^*$-algebra such that 
the above operator $\Lambda$ is an $H^*$-algebras isomorphism.
Its restriction
$\Lambda:\H\otimes\overline{\H}\rightarrow\mathbb F(\H)$ is a Hilbert algebra isomorphism
between the algebraic tensor product and the finite rank operators.
We record for further use some relations valid for $u,v,u',v'\in\H$ and $S\in\mathbb B(\H)$\,:
\begin{gather}
\label{catar}
S\lambda_{u,v}=\lambda_{Su,v}\,,\quad \lambda_{u,v}S=\lambda_{u,S^*v}\,,\quad \lambda_{u,v}\,\lambda_{u',v'}=\<u',v\>\lambda_{u,v'}\,, \\
\label{camatar}
\lambda_{u,v}^*=\lambda_{v,u}\,,\quad \Tr\(\lambda_{u,v}\)=\<u,v\>.
\end{gather}
Very often, besides the norm topology of a Hilbert algebra $\mathscr A$, there is another
finer locally convex topology.
We recall that a Fr\'echet $^*$-algebra is a $^*$-algebra $(\mathscr A,\#,^\#)$ with a Fr\'echet locally convex space topology
$\mathscr T$ such that the involution $\mathscr A\ni f\rightarrow f^\#\in \mathscr A$ is continuous and the product
$\mathscr A\times\mathscr A\ni(f,g)\rightarrow f\# g\in\mathscr A$ is separately continuous.
Then a \emph{Fr\'echet-Hilbert algebra} $(\mathscr A,\#,^\#,\mathscr T,\<\cdot,\cdot\>)$ is both a Fr\'echet $^*$-algebra and a Hilbert
algebra, the topology $\mathscr T$ being finer that the topology associated to the scalar product.

\section{Square-integrable operator-valued maps}\label{vaca}

Let us fix a complex Hilbert space $\H$, a Borel space $\Si$ with a $\si$-algebra $\mathcal M$ and a positive measure $\mu$ on $\Si$. 
The set of measurable complex-valued functions on $\Si$ will be denoted by $\mathscr M(\Si)$.
We assume that $\pi:\Si\rightarrow\mathbb B(\H)$ is a weakly measurable, almost everywhere defined map. 
One defines the sesquilinear mapping
\begin{equation}\label{bou}
\phi^\pi\equiv\phi:\H\times\H\rightarrow \mathscr M(\Si),\ \quad \phi_{u,v}(s):=\<\pi(s)u,v\>.
\end{equation}
It extends concepts such as {\it representation coefficients}, {\it wavelet transform}, {\it Short Time Fourier Transform}.

\begin{notation}
\label{cal}
\normalfont
We will use the notation 
\begin{equation}\label{gnu}
\Phi^\pi\equiv\Phi:\H\widehat\otimes\,\overline{\H}\rightarrow\L^2(\Si;\mu)\equiv \L^2(\Si).
\end{equation}
if the mapping $\phi^\pi$ admits such an {\bfi isometric} extension.
\end{notation}

\begin{remark}\label{magar}
\normalfont
The map $\Phi^\pi$ from Notation~\ref{cal} exists if and only if 
\begin{equation}\label{SQ_def}
(\forall\,u_1,u_2,v_1,v_2\in\Hc)\quad \int_\Sigma\la\pi(s)u_1,v_1\ra \la v_2,\pi(s)u_2\ra d\mu(s)=\la u_1,u_2\ra\la v_2,v_1\ra. 
\end{equation}
To achieve this equality, a renormalization of the measure $\mu$ may be used if necessary. 
Also note that it is enough to check that the equations \eqref{SQ_def} are satisfied  
for the vectors $u_1,u_2,v_1,v_2$ merely in some dense subset of~$\Hc$. 
A simple polarization argument also shows that it suffices to verify \eqref{SQ_def} for $u_1=u_2$ and $v_1=v_2$, 
that is 
\begin{equation}\label{porc}
(\forall u,v\in\Hc)\quad \int_\Si\!|\<\pi(s)u,v\>|^2d\mu(s)=\,\,\parallel\!u\!\parallel^2\,\parallel\!v\!\parallel^2.
\end{equation}
\end{remark} 

It will be convenient to make the following definition. 

\begin{definition}\label{SQ}
\normalfont
For any complex Hilbert space $\Hc$ and any measure space $(\Sigma,\Mc,\mu)$ we define $\SQ(\BB(\Hc),\mu)$ as the set of all weakly measurable, almost everywhere defined maps 
$\pi\colon\Sigma\to\BB(\Hc)$  satisfying the square-integrability condition~\eqref{SQ_def}.
\end{definition}

Before continuing, let us mention that Definition~\ref{SQ} was motivated by 
several important examples 
of operator-valued maps that satisfy the above square-integrability condition: 
\begin{itemize}
\item unitary irreducible representations of compact groups (see Corollary~\ref{irred2}); 
\item unitary irreducible representations of connected, simply connected, nilpotent Lie groups 
(see Corollary~\ref{irred3} and Proposition~\ref{WP}); 
\item the magnetic Weyl systems on $\RR^{2n}$ (see Subsection~\ref{avat}); 
\item operator calculi on locally compact abelian groups (see Proposition~\ref{meta}); 
\item localized Weyl calculus for some unitary representations of infinite-dim\-ensional Lie groups (see Subsection~\ref{lostrita}).
\end{itemize}
  
We will now obtain some simple results  which point out that the square-integra\-bil\-ity property of operator-valued maps 
should be viewed as a kind of irreducibility in the sense of representation theory,  that is, it implies the absence of nontrivial invariant subspaces. 
It is well known  that the assertions of the following proposition are  equivalent ways to describe the irreducibility property  
if the operator set $\pi(\Sigma)$ would be assumed to be closed under operator adjoints. 
As we do not assume that self-adjointness hypothesis, we will prove these assertions separately.  

\begin{proposition}\label{irred1}
Let $\pi\in\SQ(\mathbb B(\Hc),\mu)$.  
Then the following assertions hold: 
\begin{enumerate}
\item If a closed linear subspace $\Hc_0\subseteq\Hc$ 
has the property $\pi(\Sigma)\Hc_0\subseteq\Hc_0$, then either $\Hc_0=\{0\}$ or $\Hc_0=\Hc$. 
\item If the operator $T\in\mathbb B(\Hc)$ has the property $T\pi(s)=\pi(s)T$ for a.e. $s\in\Sigma$, 
then $T= z 1_{\Hc}$ for some $z\in\CC$. 
\end{enumerate}
\end{proposition}

\begin{proof}
For the first assertion, assume $\Hc_0\subsetneqq\Hc$.  Then there exists some nonzero vector $v\in\H$ with $v\perp\Hc_0$. 
Hence, by using the hypothesis $\pi(\Sigma)\Hc_0\subseteq\Hc_0$, we obtain  $v\perp\pi(s)u$ for all $s\in\Sigma$ and $u\in\Hc_0$. 
Setting in \eqref{SQ_def} $u_1=u_2=u$ and $v_1=v_2=v$, it follows that for all $u\in\Hc_0$ we have $\Vert u\Vert^2\Vert v\Vert^2=0$, 
hence necessarily $u=0$.  Consequently $\Hc_0=0$, and this concludes the proof. 

Now we prove the second assertion. First note that, for every operator $T\in\mathbb B(\Hc)$ satisfying the condition $T\pi(s)=\pi(s)T$ for a.e. $s\in\Sigma$ 
we have $\la\pi(\cdot)Tu_1,v_1\ra=\la\pi(\cdot)u_1,T^*v_1\ra$, hence by~\eqref{SQ_def}
$$(\forall u_1,u_2,v_1,v_2\in\Hc)\quad 
\la Tu_1,u_2\ra\la v_2,v_1\ra=\la u_1,u_2\ra\la v_2,T^*v_1\ra. $$
Now for  $u_1=u_2$ and $v_1=v_2$ we obtain 
$$(\forall u,v\in\Hc\setminus\{0\})\quad\frac{\la Tu,u\ra}{\Vert u\Vert^2}=\frac{\la Tv,v\ra}{\Vert v\Vert^2}=:z\in\CC.$$
Thus the numerical range of the operator $T$ consists of a single point, and $T$ is then a scalar multiple of the identity operator. 
Specifically, the above equalities imply $\la (T-z 1_{\Hc})u,u\ra=0$ for all $u\in\Hc$. 
Then, by polarization, $\la (T-z 1_{\Hc})u,v\ra=0$ for all $u,v\in\Hc$, hence eventually $T= z 1_{\Hc}$, which completes the proof. 
\end{proof}

The above proposition, 
directly implies that direct sums of square-integrable maps may not be square integrable. 
That is, if $\pi_j(\cdot)\in\SQ(\BB(\Hc_j),\mu_j)$ for $j=1,2$, then the map 
$\begin{pmatrix} \pi_1(\cdot) & 0 \\ 
0 & \pi_2(\cdot)\end{pmatrix}$ 
does not belong to $\SQ(\BB(\Hc_1\oplus\Hc_2),\mu_1\otimes\mu_2)$ unless we have either $\Hc_1=\{0\}$ or $\Hc_2=\{0\}$ 
(see however Proposition~\ref{sums} below). 

We now derive other consequences of Proposition~\ref{irred1}.    

\begin{corollary}\label{irred2}
If $\Sigma$ is a compact group with the probability Haar measure~$\mu$ 
and $\pi\colon\Sigma\to{\mathbb B}(\Hc)$ is a unitary representation, 
then $\pi\in\SQ(\mathbb B(\Hc),\mu)$ 
if and only if $\pi$ is an irreducible representation. 
\end{corollary}

\begin{proof}
It is well known that the irreducible representations of compact groups are square integrable with respect to the Haar measure, 
and the converse implication follows by Proposition~\ref{irred1}. 
\end{proof}

We also state the following corollary here for the sake of completeness of the information available so far 
on this circle of ideas, 
although its nontrivial implication depends on some aspects of representation theory of nilpotent Lie groups 
to be discussed in Subsection~\ref{Ped}.

\begin{corollary}\label{irred3}
Let $\Sigma$ be any connected, simply connected, nilpotent Lie group 
and $\pi\colon\Sigma\to{\mathbb B}(\Hc)$ be any unitary representation.  
Then $\pi$ is an irreducible representation if and only if there exists a Borel measure $\mu$ on $\Sigma$ 
for which $\pi\in\SQ(\mathbb B(\Hc),\mu)$. 
\end{corollary}

\begin{proof}
If $\pi\in\SQ(\mathbb B(\Hc),\mu)$ for some measure $\mu$, then  the representation $\pi$ is irreducible by Proposition~\ref{irred1}. 
The converse implication, including details on the construction of the measure $\mu$ in terms of the representation~$\pi$, 
is the object of Proposition~\ref{WP} below.
\end{proof}

The above Proposition~\ref{irred1} 
suggests the problem of determining the topological groups for which every unitary irreducible representation admits a measure on the group for which the representation is square integrable. 
As Corollaries \ref{irred2} and \ref{irred3} show, 
that property is shared by both the compact topological groups and the connected, simply connected nilpotent Lie groups, 
which looks somehow surprising, since these two types of groups have rather few common features. 
It would be interesting to find other examples of topological groups whose unitary irreducible representations are square-integrable with respect to suitable measures.

Despite the above irreducibility properties of square-integrable maps, 
let us note that the direct sums of such maps do have a weaker property, 
as recorded in the following observation, which is needed in the proof of Theorem~\ref{inf}.  

\begin{proposition}\label{sums}
Let $(\Sigma,\Mc,\mu)$ be any measure space. 
Let $J$ be any countable index set and for every $j\in J$ let $\Hc_j$ be any complex Hilbert space 
and $\pi_j\in\SQ(\BB(\Hc_j),\mu)$. 
Assume that for almost every $s\in\Sigma$ we have $\sup\limits_{j\in J}\Vert\pi_j(s)\Vert<\infty$. 
If we set $\Hc:=\bigoplus\limits_{j\in J}\Hc_j$, 
then the almost everywhere defined map 
$\pi:=\bigoplus\limits_{j\in J}\pi_j(\cdot)\colon\Sigma\to\BB(\Hc)$ 
is weakly measurable and has the property 
\begin{equation}\label{sums_eq1}
\int_\Sigma\vert\la\pi(s)u_1,v_1\ra \la v_2,\pi(s)u_2\ra\vert d\mu(s)
\le \Vert u_1\Vert \Vert v_1\Vert \Vert v_2\Vert \Vert u_2\Vert   
\end{equation}
for all $u_1,u_2,v_1,v_2\in\Hc$. 
\end{proposition}

\begin{proof}
Since the index set $J$ is countable and the map $\pi_j$ is weakly measurable for every $j\in J$, 
it easily follows that the map $\pi$ is in turn weakly measurable. 
For arbitrary  $j\in J$ and $u_1,u_2,v_1,v_2\in\Hc$ and we will denote their projections on $\Hc_j$ 
by $u_{1j},u_{2j},v_{1j},v_{2j}\in\Hc_j$, respectively. 
Then we have 
\allowdisplaybreaks
\begin{align}
 \int_\Sigma & \vert\la\pi(s)u_1,v_1\ra \la v_2,\pi(s)u_2\ra\vert d\mu(s)
\le 
 \sum_{j,k\in J}\int_\Sigma\vert\la\pi_j(s)u_{1j},v_{1j}\ra \la v_{2k},\pi_k(s)u_{2k}\ra\vert d\mu(s) \nonumber\\
\le 
& \sum_{j,k\in J}\Bigl(\int_\Sigma\vert\la\pi_j(s)u_{1j},v_{1j}\ra\vert^2 d\mu(s)\Bigr)^{1/2}  
\Bigl(\int_\Sigma\vert\la v_{2k},\pi_k(s)u_{2k}\ra\vert^2 d\mu(s)\Bigr)^{1/2} \nonumber\\
= 
& \sum_{j,k\in J}\Vert u_{1j}\Vert \Vert v_{1j}\Vert \Vert v_{2k}\Vert u_{2k}\Vert \nonumber\\
= 
&\Bigl(\sum_{j\in J}\Vert u_{1j}\Vert \Vert v_{1j}\Vert \Bigr)
 \Bigl(\sum_{k\in J}\Vert \Vert v_{2k}\Vert u_{2k}\Vert\Bigr) \nonumber\\ 
\le & 
\Bigl(\sum_{j\in J}\Vert u_{1j}\Vert^2\Bigr)^{1/2} 
\Bigl(\sum_{j\in J}\Vert v_{1j}\Vert^2\Bigr)^{1/2} 
\Bigl(\sum_{k\in J}\Vert v_{2k}\Vert^2\Bigr)^{1/2} 
\Bigl(\sum_{k\in J}\Vert u_{2k}\Vert^2\Bigr)^{1/2} \nonumber\\ 
=& 
\Vert u_1\Vert \Vert v_1\Vert \Vert v_2\Vert \Vert u_2\Vert \nonumber
\end{align}
and this concludes the proof. 
\end{proof}

We now draw a consequence that will be needed in the proof of Theorem~\ref{inf} below. 

\begin{corollary}\label{mult}
Let $(\Sigma,\Mc,\mu)$ be any measure space and $\Kc$ be any separable complex Hilbert space. 
If $\pi_0\in\SQ(\BB(\Hc_0),\mu)$, 
then $\pi(\cdot):=\pi_0(\cdot)\otimes\id_{\Kc}\colon\Sigma\to\BB(\Hc_0\widehat{\otimes}\Kc)$ 
is a weakly measurable map that satisfies \eqref{sums_eq1} 
for all $u_1,u_2,v_1,v_2\in\Hc_0\widehat{\otimes}\Kc$.
\end{corollary}

\begin{proof}
Let $J$ be the index set of any orthonormal basis of~$\Kc$, 
so that we may assume $\Kc=\ell^2(J)$. 
If we set $\Hc_j:=\Hc_0$ and $\pi_j:=\pi_0$ for all $j\in J$, 
then we may apply Proposition~\ref{sums}, and the conclusion follows. 
\end{proof}

We discuss below a few operations on operator maps satisfying our square-integrability condition required in 
Definition~\ref{SQ}. 
In particular, these operations will provide as many methods to construct new square-integrable maps out of other maps satisfying the same hypothesis. 

\subsection{Tensor products of square-integrable maps} 

\begin{proposition}\label{SQ_prod}
For $j=1,2$,  let $\Hc_j$ be any complex Hilbert space, $(\Sigma_j,\Mc_j,\mu_j)$ be 
any $\sigma$-finite measure space, and let $\pi_j\in\SQ(\BB(\Hc_j),\mu_j)$. 
If we define 
$$\pi_1\otimes\pi_2\colon\Sigma_1\times\Sigma_2\to\BB(\Hc_1\widehat\otimes\Hc_2),\quad 
(\pi_1\otimes\pi_2)(s_1,s_2):=\pi_1(s_1)\otimes\pi_2(s_2),$$
then $\pi_1\otimes\pi_2\in\SQ(\BB(\Hc_1\widehat\otimes\Hc_2),\mu_1\otimes\mu_2)$. 
\end{proposition}

\begin{proof}
For any $f_j\in L^2(\Sigma_j,\mu_j)$ with $j=1,2$, define 
$f_1\otimes f_2\in L^2(\Sigma_1\times\Sigma_2,\mu_1\otimes\mu_2)$ by 
$(f_1\otimes f_2)(x_1,x_2):=f_1(x_1)f_2(x_2)$ almost everywhere. 
Recall that the bilinear map 
$\L^2(\Sigma_1,\mu_1)\times L^2(\Sigma_2,\mu_2)\to L^2(\Sigma_1\times\Sigma_2,\mu_1\otimes\mu_2)$,
$ (f_1,f_2)\mapsto f_1\otimes f_2$, 
gives rise to a unitary operator 
$$V\colon \L^2(\Sigma_1,\mu_1)\widehat\otimes L^2(\Sigma_2,\mu_2)\to\L^2(\Sigma_1\times\Sigma_2,\mu_1\otimes\mu_2).$$ 
In fact, using the Fubini theorem, the operator $V$ is an isometry, and for proving that it is also surjective, 
we may assume $\mu_j(\Sigma_j)<\infty$ for $j=1,2$. 
Then it suffices to show that the set 
$${\mathcal Q}:=\{A\subseteq \Sigma_1\times\Sigma_2\mid \chi_A\in \Ran V\}$$ 
(where we have denoted by $\chi_A$ the characteristic function of the set $A$) 
contains the $\sigma$-ring ${\mathcal Q}_0$ of all the $\mu_1\otimes\mu_2$-measurable subsets of $\Sigma_1\times\Sigma_2$, 
since $\{\chi_A\mid A\in{\mathcal Q}_0\}$ spans a dense linear subspace of $L^2(\Sigma_1\times\Sigma_2,\mu_1\otimes\mu_2)$.  
Recall that ${\mathcal Q}_0$ is the $\sigma$-ring generated by the sets $A_1\times A_2$ for all measurable sets $A_j\subseteq \Sigma_j$ 
with  $j=1,2$. 
It suffices to note that ${\mathcal Q}$ is a ring  
(i.e., it is closed under finite unions and differences),  
${\mathcal Q}$ is closed under countable unions of increasing sequences by Lebesgue's dominated convergence theorem since we assumed the measures $\mu_1$ and $\mu_2$ to be finite, 
and moreover ${\mathcal Q}$ contains all the above mentioned sets $A_1\times A_2$, 
hence ${\mathcal Q}_0\subseteq{\mathcal Q}$ by the monotone class theorem 
\cite[Ch. I, \SS6, Th. B]{Ha50}.

Then for $j=1,2$, 
by using the hypothesis $\pi_j\in\SQ(\BB(\Hc_j),\mu_j)$ along with~\eqref{gnu}, 
we obtain the isometry 
$$\Phi^{\pi_j}\colon\Hc_j\widehat\otimes\overline{\Hc}_j\to\L^2(\Sigma_j,\mu_j), 
\quad  \Phi^{\pi_j}(u\otimes v)=\la\pi_j(\cdot)u,v\ra.$$
As the Hilbertian tensor product of two isometries is again an isometry, 
we obtain that the operator 
$$\Phi^{\pi_1}\otimes \Phi^{\pi_2}\colon
(\Hc_1\widehat\otimes\overline{\Hc}_1)\widehat\otimes (\Hc_2\widehat\otimes\overline{\Hc}_2)\to
\L^2(\Sigma_1,\mu_1)\widehat\otimes L^2(\Sigma_2,\mu_2)$$
is an isometry. 
On the other hand, for $j=1,2$, all $u_j,v_j\in\Hc_j$, and almost all $s_j\in\Sigma_j$ we have 
$$\begin{aligned}
((\Phi^{\pi_1}\otimes \Phi^{\pi_2})((u_1\otimes v_1)\otimes & (u_2\otimes v_2)))(s_1,s_2) \\
& =\la\pi(s_1)u_1,v_1\ra \la\pi(s_2)u_2,v_2\ra \\
&=\la((\pi_1\otimes\pi_2)(s_1,s_2))(u_1\otimes u_2),v_1\otimes v_2\ra \\
&=\phi^{\pi_1\otimes\pi_2}_{u_1\otimes u_2,v_1\otimes v_2}(s_1,s_2). 
\end{aligned}$$
Now, by composing the isometry $\Phi^{\pi_1}\otimes \Phi^{\pi_2}$ with the flip unitary operator 
$$(\Hc_1\widehat\otimes\overline{\Hc}_1)\widehat\otimes (\Hc_2\widehat\otimes\overline{\Hc}_2)\to
(\Hc_1\widehat\otimes\Hc_2)\widehat\otimes \overline{\Hc_1\widehat\otimes\Hc_2},\ 
u_1\otimes v_1\otimes u_2\otimes v_2\mapsto u_1\otimes u_2\otimes v_1\otimes v_2 $$
and with the above unitary operator $V$, 
it follows that the sesquilinear mapping 
$$\phi^{\pi_1\otimes\pi_2}\colon(\Hc_1\widehat\otimes\Hc_2)\times(\Hc_1\widehat\otimes\Hc_2)
\to\mathscr M(\Sigma_1\times\Sigma_2)$$ 
(see \eqref{bou}) 
gives rise to the isometry 
$\Phi^{\pi_1\otimes\pi_2}=\Phi^{\pi_1}\otimes \Phi^{\pi_2}$. 
This  
implies that $\pi_1\otimes\pi_2\in\SQ(\BB(\Hc_1\widehat\otimes\Hc_2),\mu_1\otimes\mu_2)$, 
and the proof is complete. 
\end{proof}

\subsection{Compressions of square-integrable maps} 
  
\begin{proposition}\label{SQ_push}
For $j=1,2$ let $\Hc_j$ be any complex Hilbert space, $(\Sigma_j,\Mc_j,\mu_j)$ be any measure space, and let $\pi_j\colon\Sigma_j\to \BB(\Hc_j)$. 
Assume that $p\colon\Sigma_2\to\Sigma_1$ is a measurable map satisfying the condition $p_*(\mu_2)=\mu_1$ 
and $\iota\colon\Hc_1\to\Hc_2$ is a linear isometry for which $\pi_1\circ p=\iota^*\pi_2(\cdot)\iota$ a.e.\ on~$\Sigma_2$. 

If $\pi_2\in\SQ(\BB(\Hc_2),\mu_2)$, then also $\pi_1\in\SQ(\BB(\Hc_1),\mu_1)$ and we have the commutative diagram 
$$\begin{CD} 
\L^2(\Sigma_1,\mu_1) @>{f\mapsto f\circ p}>> \L^2(\Sigma_2,\mu_2) \\
@A{\Phi^{\pi_1}}AA @AA{\Phi^{\pi_2}}A \\
\Hc_1\widehat\otimes\overline{\Hc}_1 @>{\iota\otimes\iota}>> \Hc_2\widehat\otimes\overline{\Hc}_2
\end{CD}$$
whose arrows are isometries.  
\end{proposition}

\begin{proof}
By using the hypothesis, for arbitrary $u_1,v_1\in\Hc_1$ we obtain  
\allowdisplaybreaks 
\begin{align}
\int_{\Sigma_1}\vert\la\pi_1(s_1)u_1,v_1\ra\vert^2 d\mu_1(s_1)
&= \int_{\Sigma_2}\vert\la\pi_1(p(s_2))u_1,v_1\ra\vert^2 d\mu_2(s_2) \nonumber\\
&= \int_{\Sigma_2}\vert\la\iota^*\pi_2(s_2)\iota (u_1),v_1\ra\vert^2 d\mu_2(s_2) \nonumber\\
&= \int_{\Sigma_2}\vert\la\pi_2(s_2)\iota(u_1),\iota(v_1)\ra\vert^2 d\mu_2(s_2) \nonumber\\
&=\Vert \iota(u_1)\Vert^2\Vert\iota(v_1)\Vert^2 =\Vert u_1\Vert^2\Vert v_1\Vert^2 \nonumber
\end{align}
 which shows that $\pi_1\in\SQ(\BB(\Hc_1),\mu_1)$. 
The assertion on the commutative diagram is then clear, and this concludes the proof. 
\end{proof}

\section{The $H^*$-algebra $\mathscr B_2(\Si)$ and larger symbol spaces}\label{vacuta}

We place ourselves in the setting of Section~\ref{vaca}; 
in particular we are given a map $\pi:\Si\to\mathbb B(\H)$ belonging to $SQ(\mathbb B(\H),\mu)$. 
We do not  assume the family $\{\phi_{u,v}\equiv\Phi(u\otimes v)\mid u,v\in\H\}$  to be total in $\L^2(\Si)$. 
For instance, this property fails if $\pi$ is any irreducible representation of any compact group $\Si\ne\{\1\}$, 
since we then have $\dim\Hc<\infty$ and the Peter-Weyl decomposition of $L^2(\Sigma)$ 
(see also Corollary~\ref{irred2}).  

Consequently, we need to introduce the closed subspace 
$$\mathscr B_2(\Si):=\Phi\(\H\widehat\otimes\,\overline{\H}\)\subseteq \L^2(\Si)$$ 
which
is unitarily equivalent to $\H\widehat\otimes\,\overline{\H}$
and hence with $\mathbb B_2(\H)$. 
This space $\mathscr B_2(\Si)$ is the closure in $\L^2(\Si)$ of the subspace
$\Phi(\H\!\otimes\overline{\H})$.
Clearly, there is a unitary operator 
\begin{equation}\label{fenix}
\Pi:=\Lambda\circ\Phi^{-1}:\mathscr B_2(\Si)\rightarrow\mathbb B_2(\H)
\end{equation}
uniquely determined by
\begin{equation}\label{capra}
\Pi\(\phi_{u,v}\)=\lambda_{u,v}=\<\cdot,v\>u\,,\quad\ \forall\,u,v\in\H
\end{equation}
and satisfying
\begin{equation*}
\Tr\[\Pi(f)\Pi(g)\!^*\]=\<f,g\>_{(\Si)}:=\int_\Si\! f(s)\,\overline{g(s)}d\mu(s)\,,\
\quad\forall\,f,g\in \mathscr B_2(\Si).
\end{equation*}
For the sake of clarity we also note the commutative diagram 
$$\xymatrix{
\Hc\widehat{\otimes}\overline{\Hc} \ar[r]^{\Phi} \ar[d]_{\Lambda} & {\mathscr B_2(\Sigma)} \ar[dl]^{\Pi}\\
{\mathbb B_2(\Hc)} &  
}
$$
whose arrows are isomorphisms of $H^*$-algebras, and in particular unitary operators.  

\begin{remark}\label{graur}
\normalfont
The above commuting diagram can be connected with previous con\-struc\-tions. 
For example, in the context of Proposition \ref{SQ_prod}, 
we can use $\pi_j\in SQ(\mathbb B(\H_j),\mu_j)$ to construct the $H^*$-algebra $\mathscr B_2(\Si_j)$ 
and the map $\Pi_j$, while $\pi_1\otimes\pi_2$ serves in the same way to construct 
the $H^*$-algebra $\mathscr B_2(\Si_1\times\Si_2)$ and the map $\Pi$. 
As a direct consequence of Proposition \ref{SQ_prod}, 
$\mathscr B_2(\Si_1\times\Si_2)$ can be identified with $\mathscr B_2(\Si_1)\widehat\otimes\mathscr B_2(\Si_2)$ 
and $\Pi$ with $\Pi_1\otimes\Pi_2$. 
And in the setting of Proposition \ref{SQ_push} 
one has $f\in\mathscr B_2(\Si_1)$ if and only if 
$f\circ p\in\Phi^{\pi_2}(\iota\H_1\widehat\otimes\,\overline{\iota\H_1})\subset\mathscr B_2(\Si_2)$, 
and then $\Pi_1(f)=\iota^*\Pi_2(f\circ p)\iota$. 
If $\iota$ is unitary, 
the two $H^*$-algebras $\mathscr B_2(\Si_1)$ and $\mathscr B_2(\Si_2)$ 
are isomorphic through the transformation $f\mapsto f\circ p$.
\end{remark}

\subsection{Basic properties of $\Pi$}

\begin{proposition}\label{lama}
For any $f\in\mathscr B_2(\Si)$ one has in weak sense
\begin{equation}\label{guanaco}
\Pi(f)=\int_\Si\! f(s)\,\pi(s)^*d\mu(s)\,,\quad\Pi(f)^*=\int_\Si\! \overline{f(s)}\,\pi(s)d\mu(s).
\end{equation}
\end{proposition}

\begin{proof}
If $u,v\in\H$ one has
$$
\begin{aligned}
\<\Pi(f)u,v\>=&\,\Tr\[\lambda_{\Pi(f)u,v}\]=\,\Tr\[\Pi(f)\,\lambda_{u,v}\,\]\\
=&\,\Tr\[\Pi(f)\,\Pi\(\phi_{v,u}\)^*\,\]=\<f,\phi_{v,u}\>_{(\Si)}\\
=&\int_\Si\!f(s)\<\pi(s)^*u,v\>d\mu(s) .
\end{aligned}
$$
Then the second formula follows from the first one.
\end{proof}

The next simple corollary is needed in the proof of Proposition~\ref{meta}.

\begin{corollary}\label{adjoint}
The adjoint of the isometry $\Phi\circ\Lambda^{-1}\colon\BB_2(\Hc)\to\L^2(\Sigma)$ 
is given by the weakly convergent integral 
$$(\Phi\circ\Lambda^{-1})^*f=\int_\Sigma f(s)\pi(s)^*d\mu(s)$$ 
for every $f\in\L^2(\Sigma)$.  
\end{corollary}

\begin{proof}
One has $\Phi\circ\Lambda^{-1}=\Pi^{-1}= \Pi^\ast$, since $\Pi$ is unitary. 
The assertion follows now by Proposition~\ref{lama}.  
\end{proof}

By transport of structure one defines a composition law and an involution:
$$
\begin{aligned}
\star:\mathscr B_2(\Si)\times \mathscr B_2(\Si)\rightarrow \mathscr B_2(\Si), & \quad f\star g:=\Pi^{-1}\!\[\Pi(f)\Pi(g)\],\\
^\star:\mathscr B_2(\Si)\rightarrow \mathscr B_2(\Si), & \quad f^\star:=\Pi^{-1}\!\[\Pi(f)\!^*\].
\end{aligned}
$$
Therefore  
$\mathscr B_2(\Si)$ {\it is an $H^*$-algebra and $\Pi:\mathscr B_2(\Si)\rightarrow\mathbb B_2(\H)$
is an $H^*$-algebra isomorphism}. 
Thus one has for all $f,g,h\in \mathscr B_2(\Si)$
\begin{equation*}
\begin{aligned}
\<f^{\star},g^{\star}\>_{(\Si)} &=\<g,f\>_{(\Si)}, \\
\<f\star g,h\>_{(\Si)} & =\<f,h\star g^{\star}\>_{(\Si)}=\<g,f^{\star}\!\star h\>_{(\Si)}.
\end{aligned}
\end{equation*}
The Hilbert subalgebra $\mathscr B_1(\Si):=\mathscr B_2(\Si)\,\td\, \mathscr B_2(\Si)$ is dense in
$\mathscr B_2(\Si)$ and for every $g\in \mathscr B_2(\Si)$ the maps
\begin{equation*}
\mathscr B_2(\Si)\ni f\mapsto 
g\td f\in \mathscr B_2(\Si)\,,\ \quad
\mathscr B_2(\Si)\ni f\mapsto 
f\td g\in \mathscr B_2(\Si)
\end{equation*}
are continuous. 

We notice for further use the relations
\begin{align}\label{ratoi}
\<\phi_{u_1,v_1},\phi_{u_2,v_2}\>_{(\Si)} & =\<u_1,u_2\>\<v_2,v_1\>\\
\phi_{u_1,v_1}\!\star\phi_{u_2,v_2} & =\<u_2,v_1\>\phi_{u_1,v_2},\quad\ \phi_{u,v}^\star=\phi_{v,u},  \label{berbec}
\end{align}
valid for every $u,u_1,u_2,v,v_1,v_2\in\H$, as well as
\begin{equation*}
f\td\phi_{u,v}\td g=\phi_{\Pi(f)u,\Pi(g)^*v}\,,\quad\ \forall\, f,g\in \mathscr B_2(\Si),\ u,v\in\H.
\end{equation*}
In particular, if $\,\p\!u\!\p\,=1$, then $\phi_{u,u}$ is a self-adjoint projection, represented by $\Pi$
as the rank-one operator $\lambda_{u,u}$.
Also notice that $\mathscr B_1(\Si)$ and $\mathscr B_2(\Si)$ are Banach $^*$-algebras, 
where $\mathscr B_1(\Si)$ is endowed with the norm for which the linear bijection 
$\Pi\colon \mathscr B_1(\Si)\to \mathbb B_1(\H)$ 
is an isometry.

\subsection{Extensions and the explicit form of the composition law}\label{platica}

In this subsection we extend some of the above maps to larger symbol spaces. 
We define a new norm
\begin{equation*}
\p\cdot\p_{\mathscr B(\Si)}\,:\mathscr B_2(\Si)\rightarrow\mathbb R_+\,,\quad\
\p\!f\!\p_{\mathscr B(\Si)}\,:=\,\p\!\Pi(f)\!\p_{\mathbb B(\H)}.
\end{equation*}
The completion of $\mathscr B_2(\Si)$ under this norm is a $C^*$-algebra $\mathscr B_\infty(\Si)$ containing $\mathscr B_2(\Si)$ as a dense $^*$-ideal (this one also endowed with the stronger Hilbert topology). 
Clearly $\Pi$ extends to a $C^*$-algebraic monomorphism $\Pi:\mathscr B_\infty(\Si)\rightarrow\mathbb B(\H)$ with range $\Pi\[\mathscr B_\infty(\Si)\]=\mathbb B_\infty(\H)$,
the ideal of all compact operators in $\H$. 
Then we denote by $\mathscr B(\Si)$ the multiplier $C^*$-algebra \cite{WO}
of $\mathscr B_\infty(\Si)$, which is isomorphic by a canonical extension of $\Pi$ with $\mathbb B(\H)$\,,
which can in turn be identified with the multiplier $C^*$-algebra of $\mathbb B_\infty(\H)$. 
We keep the same notations $\td$ and $^{\td}$ for the composition law and the involution on $\mathscr B(\Si)$. 
Based on the constructions above, the elements of $\mathscr B_1(\Si)$ ($\mathscr B_2(\Si)$, $\mathscr B_\infty(\Si)$) will be called respectively {\it trace-class (Hilbert-Schmidt, compact) symbols}. 
For the elements of 
$\mathscr B(\Si)$, to eliminate any possible confusion, we reserve the term {\it operator-bounded symbols}.

The spaces $\mathscr B_q(\Si)$\,, $q=1,2,\infty$ are still $^*$-ideals in $\mathscr B(\Si)$ and the scalar product
$\<\cdot,\cdot\>_{\mathscr B_2(\Si)}$ can be ``extended'' to sesquilinear forms
$$
\<\cdot,\cdot\>_{(\Si)}:\mathscr B_1(\Si)\times\mathscr B(\Si)\rightarrow\C\,,\quad
\<\cdot,\cdot\>_{(\Si)}:\mathscr B(\Si)\times \mathscr B_1(\Si)\rightarrow\C.
$$
For this one simply sets $\,\<f,g\>_{(\Si)}:=\Tr\[\Pi(f)\Pi(g)^*\]\,$ (definitions by approximation are also available).
Note for $f\in\mathscr B_1(\Si)$ and $g\in\mathscr B(\Si)$ the inequality
\begin{equation*}
|\<f,g\>_{(\Si)}|\le\,\p\!f\!\p_{\mathscr B_1(\Si)}\,\p\!g\!\p_{\mathscr B(\Si)}\,\equiv\,\,\p\!
\Pi(f)\!\p_{\mathbb B_1(\H)}\,\p\!\Pi(g)\!\p_{\mathbb B(\H)}.
\end{equation*}
Due to the cyclicity of the trace, if $f,g,h\in\mathscr B(\Si)$ and one of them belongs to $\mathscr B_1(\Si)$
(or two of them belong to $\mathscr B_2(\Si)$)\,, one has
\begin{equation*}
\<f\td g,h\>_{(\Si)}=\<f,h\td g^{\td}\>_{(\Si)}=\<g,f^{\td}\!\td h\>_{(\Si)}.
\end{equation*}

Let us put 
$$(\forall s\in\Sigma)\quad e_s:=\Pi^{-1}\[\pi(s)^*\]\in\mathscr B(\Si)$$
hence $\phi_{u,v}(s)=\<u,\Pi(e_s)v\>$ for all $u,v$ and $s$ and
\begin{equation}\label{arici}
\pi(s)^*=\Pi\(e_s\),\quad\ \pi(s)=\Pi\(e^{\td}_s\).
\end{equation}

\begin{proposition}\label{guvid}
For every $f\in \mathscr B_1(\Si)$ one has
\begin{equation}\label{dromader}
\<f,e_s\>_{(\Si)}=f(s)\,,\quad\<f,e_s^{\td}\>_{(\Si)}=\overline{f^{\td}(s)}\,, \quad\ \mu-a.e.\ s\in\Si.
\end{equation}
\end{proposition}

\begin{proof}
By a direct computation using (\ref{capra}), (\ref{catar}) and (\ref{camatar}) one gets for $u,v\in\H,\ s\in\Si$
$$
\<\phi_{u,v},e_s\>_{\Si}=\Tr\[\Pi(\phi_{u,v})\pi(s)\]=\Tr\[\lambda_{u,v}\pi(s)\]=\Tr\[\lambda_{u,\pi(s)^*v}\]=\phi_{u,v}(s)\,,
$$
so the same is true for $\phi_{u,v}$ replaced by any element of $\Phi(\H\otimes\overline{\H})$\,,
which is dense in $\mathscr B_1(\Si)$.

Assume now that a sequence $\{g_n\}_{n\in\mathbb N}\subset\Phi(\H\otimes\overline{\H})$ converges to
$f\in\mathscr B_1(\Si)$ with respect to the trace norm. Then for each $s\in\Si$
$$
\begin{aligned}
|\<f,e_s\>_{(\Si)}-\<g_n,e_s\>_{(\Si)}|&=|\,\Tr\[\Pi(f-g_n)\pi(s)\]|\\
&\le\,\p\!\Pi(f-g_n)\!\p_{\mathbb B_1(\H)}\,\p\!\pi(s)\!\p_{\mathbb B(\H)}\\
&=\p\!f-g_n\!\p_{\mathscr B_1(\Si)}\p\!\pi(s)\!\p_{\mathbb B(\H)}\underset{n\rightarrow\infty}{\longrightarrow}0
\end{aligned}
$$
recalling that $\Pi\colon \mathscr B_1(\Si)\to \mathbb B_1(\H)$ 
is an isometry by the definition of the norm of $\mathscr B_1(\Si)$. 
Since convergence in $\mathscr B_1(\Si)$ implies convergence in $\L^2(\Si)$ which, in its turn,
implies $\mu$-almost everywhere convergence of a subsequence, there is a $\mu$-negligible set $M\subset\Si$ and
a subsequence $\{g_{n_k}\}_{k\in\mathbb N}$ such that for every $s\in\Si\setminus M$
$$
f(s)=\lim_{k\rightarrow\infty} g_{n_k}(s)=\lim_{k\rightarrow\infty} \<g_{n_k},e_s\>_{(\Si)}=
\<f,e_s\>_{(\Si)}.
$$
Then for $s$ belonging to the same set $s\in\Si\setminus M$ one has
$$
\<f,e_s^{\td}\>_{(\Si)}=\<e_s,f^\star\>_{(\Si)}=\overline{\<f^\star,e_s\>}_{(\Si)}=\overline{f^{\td}(s)}.
$$
\end{proof}

\begin{corollary}
For every $f,g\in \mathscr B_1(\Si)$ one has
\begin{equation*}
\int_\Si\!\la f,e_s\ra_{(\Si)}\la e_s,g\ra_{(\Si)}\,d\mu(s)=\la f,g\ra_{(\Si)}.
\end{equation*}
\end{corollary}

\begin{proof}
Follows immediately from Proposition~\ref{guvid}.
\end{proof}

Now we can compute the symbol of a trace-class operator.

\begin{corollary}\label{fitofag}
For $\,T\in\mathbb B_1(\H)\subset\mathbb B_2(\H)$ one has
\begin{equation}\label{crapcean}
\[\Pi^{-1}(T)\]\!(s)=\Tr\[T\pi(s)\],\quad\ \mu\text{-a.e. }s\in\Si.
\end{equation}
\end{corollary}

\begin{proof}
Let us denote for the moment by $\Pi^{(-1)}$ the mapping defined at (\ref{crapcean}).
It is enough to show that $\Pi^{(-1)}\[\Pi(f)\]=f$ hold $\mu$-a.e. for every $f$ belonging $\mathscr B_1(\Si)$.

But for $\mu$-almost every $s\in\Si$\,, one has by (\ref{arici}) and (\ref{dromader})
$$
\(\Pi^{(-1)}\[\Pi(f)\]\)\!(s)=\Tr\[\Pi(f)\pi(s)\]=\<f,e_s\>_{(\Si)}=f(s).
$$
\end{proof}

\begin{corollary}\label{dalailama}
For each $S\in\mathbb B_1(\H)$ and each $f\in\mathscr B_2(\Si)$ one has
\begin{align}
\Tr[\Pi(f)S] & =\int_\Si\! f(s)\,\Tr\left[\pi(s)^*\!S\right]d\mu(s),\nonumber \\
\label{magaruz}
\Tr[\Pi(f)^*\!S] & =\int_\Si\! \overline{f(s)}\,\Tr\left[\pi(s)S\right]d\mu(s).
\end{align}
\end{corollary}

\begin{proof}
One uses the definitions, the fact that $\Pi$ is unitary and formula (\ref{crapcean}):
$$
\begin{aligned}
\Tr[\Pi(f)S]=&\<\Pi(f),\Pi\[\Pi^{-1}(S^*)\]\>_{\mathbb B_2(\H)}=\<f,\Pi^{-1}(S^*)\>_{(\Sigma)}\\
=&\int_\Si\!f(s)\overline{\[\Pi^{-1}(S^*)\](s)}d\mu(s)=\int_\Si\!f(s)\overline{\Tr\[S^*\pi(s)\]}d\mu(s)\\
=&\int_\Si\!f(s)\Tr\left[\pi(s)^*\!S\right]d\mu(s).
\end{aligned}
$$
The relation \eqref{magaruz} follows similarly.
\end{proof}

Corollary~\ref{dalailama} (which can be alternatively derived from the equation \eqref{guanaco})
reenforces Proposition \ref{lama}, recovered by taking $S$ to be a rank one operator.

\begin{remark}\label{rechinas}
{\rm One can also justify for every $f\in\mathscr B_2(\Si)$ the relations
\begin{equation*}
f=\int_\Si\! f(s)\,e_sd\mu(s)\,,\quad\ f^\star=\int_\Si\! \overline{f(s)}\,e_s^\star d\mu(s)\,;
\end{equation*}
for example, if $g\in\mathscr B_1(\Si)$, then
$$
\int_\Si\! f(s)\<e_s,g\>_{(\Si)}d\mu(s)=\int_\Si\! f(s)\overline{g(s)}d\mu(s)=\<f,g\>_{(\Si)}.
$$
In general $\mathscr B_2(\Si)$ is not a reproducing kernel Hilbert space; the symbols $e_s$ rarely belong to $\mathscr B_2(\Si)$.
}
\end{remark}

We give now explicit formulae for the algebraic structure. 
In the following statement we use the fact that for any complex Hilbert space $\H$, 
the Banach algebra of trace-class operators $\mathbb B_1(\H)$ has right approximate units.   
For instance, the family of orthogonal projections onto finite-dimensional subspaces of $\H$ 
is even a two-sided approximate unit of $\mathbb B_1(\H)$. 
This follows by \cite[Th. 6.3, Ch. III]{GK69} on separable Hilbert spaces, 
and then it extends to arbitrary Hilbert spaces, writing every operator in $\mathbb B_1(\H)$ 
as a linear combination of self-adjoint operators and using the fact that 
the closure of the range of any compact self-adjoint operator is separable. 

\begin{theorem}\label{tipar}
Let $\{S_j\mid j\in J\}$ be any right approximate unit  in $\mathbb B_1(\H)$.
\begin{enumerate}
\item
If $f\in\mathscr B_1(\Si)$, for $\mu$-almost every $r\in\Si$ one has
\begin{equation*}
f^\star(r)=\lim_{j}\int_\Si\!\Tr\[\pi(r)\pi(s)S_j\]\overline{f(s)}d\mu(s).
\end{equation*}
\item
If $f,g\in\mathscr B_2(\Si)$, for $\mu$-almost every $r\in\Si$ one has
\begin{equation*}
(f\star g)(r)=\lim_{j}\int_\Si\int_\Si\!\Tr\[\pi(s)^*\pi(t)^*\pi(r)S_j\]f(s)g(t)d\mu(s)d\mu(t).
\end{equation*}
\end{enumerate}
\end{theorem}

\begin{proof}
Both computations rely on Corollaries \ref{fitofag} and \ref{dalailama}. One has for $\mu$-almost every $r\in\Si$
$$
\begin{aligned}
f^\star(r)=&\,\Tr\left[\Pi(f)^*\pi(r)\right]=\lim_{j}\Tr\left[\Pi(f)^*\pi(r)S_j\right]\\
=&\lim_{j}\int_\Si\!\Tr\[\pi(s)\pi(r)S_j\]\overline{f(s)}d\mu(s)
\end{aligned}
$$
and
$$
\begin{aligned}
(f\star g)(r)=&\,\Tr\left[\Pi(f)\Pi(g)\pi(r)\right]=\lim_{j}\Tr\left[\Pi(f)\Pi(g)\pi(r)S_j\right]\\
=&\lim_{j}\int_\Si\!f(s)\Tr\[\pi(s)^*\Pi(g)\pi(r)S_j\]d\mu(s)\\
=&\lim_{j}\int_\Si\!f(s)d\mu(s)\!\int_\Si\!g(t)\Tr\[\pi(s)^*\pi(t)^*\pi(r)S_j\]d\mu(t)
\end{aligned}
$$
and this concludes the proof.
\end{proof}

\begin{remark}\label{breb}
\normalfont 
Let us assume that $\Si$ is a locally compact space and we have 
$\Vert\pi(s)\Vert_{\mathbb B(\H)}\le C<\infty$ for all $s\in\Si$. 
Let $\mathscr R(\Si)$ be the Banach space of all Radon bounded complex measures on $\Si$, seen alternatively
both as functions on the Borel sets of $\Si$ and as elements of the topological anti-dual of $C_0(\Si)$. 
Using the Hahn-Banach theorem, one can easily find a norm-preserving extension of every
$\rho\in\mathscr R(\Si)$ 
to an anti-linear continuous functional 
$\rho:BC(\Si)\rightarrow\mathbb C$, 
where $BC(\Sigma)$ denotes the Banach space of all bounded continuous functions on $\Sigma$. 
We will use the notation $\<\!\<\,\rho,f\>\!\>=\int_\Si\overline{f}\,d\rho$ for this "duality"
(linear in $\rho$ and anti-linear in $f$). 
On $\mathscr R(\Si)$ the usual norm of an anti-dual coincides with the measure norm
expressed as the total variation applied to the entire space $\Si$. Since for every $u,v\in\H$ one has
$\phi_{v,u}\in\mathscr B_2(\Si)\cap BC(\Si)$, one can define $\Pi(\rho)\in\mathbb B(\H)$ in weak sense by
\begin{equation*}
\<\Pi(\rho)u,v\>:=\int_\Si \<\pi(s)^*u,v\>d\rho(s)=\<\!\<\rho,\overline{\<\pi(\cdot)^*u,v\>}\>\!\>=\<\!\<\,\rho,\phi_{v,u}\>\!\>.
\end{equation*}
Now it is also obvious that $e_t$ coincides with the Dirac measure concentrated in $t$ and that if $\rho$ has a density
$g$ with respect to the initial measure $\mu$, then $\Pi(\rho)=\Pi(g)$.
The estimate
$$
\p\!\Pi(\rho)\!\p_{\mathbb B(\H)}\,\le\,\,\p\!\rho\!\p_{\mathscr R(\Si)}\sup_{s\in\Si}\p\!\pi(s)\!\p_{\mathbb B(\H)}
$$
is easy and certifies that $\L^1(\Si,\mu)\subset\mathscr R(\Si)\subset\mathscr B(\Si)$.
\end{remark}

\section{Fr\'echet-Hilbert algebras and their Gelfand triples}\label{graur_sect}

In most of the applications there is some supplementary structure which can be used to enrich and enlarge the formalism.
Let $\G$ be a Fr\'echet space continuously and densely embedded in $\H$; set $\alpha:\G\rightarrow\H$ for the embedding.
We are going to show that such an extra data generates many new useful objects, even if we do not require
$\pi(s)\G\subset\G$ for $s\in\Si$. Assuming that $\G$ is nuclear would simplify the overall picture, but we also want to cover the case of Banach spaces.

\begin{lemma}\label{tapir}
The projective tensor product $\G\widehat\otimes_p\,\overline\G$ is a Fr\'echet space continuously and densely embedded in
the Hilbert space $\H\widehat\otimes\,\overline\H$.
\end{lemma}

\begin{proof}
Clearly, $\G\widehat\otimes_p\,\overline\G$ is a Fr\'echet space. It
is enough to embed it continuously and densely into $\H\widehat\otimes_p\overline \H$.
The canonical mapping $\alpha\otimes\alpha:\G\otimes\overline\G\rightarrow\H\otimes\overline\H$ is linear and continuous
when on both tensor products we put the corresponding projective topologies \cite[Prop. 43.6]{Tr},
so it extends to a linear continuous map
$\alpha\widehat\otimes_p\,\alpha:\G\widehat\otimes_p\,\overline\G\rightarrow\H\widehat\otimes_p\overline\H$.
This map obviously has a dense range, and it is injective by \cite[Ex. 43.2]{Tr}.
By composing with the canonical injection $\H\widehat\otimes_p\overline \H\hookrightarrow\H\widehat\otimes\,\overline \H$\,,
we get the linear continuous map $\mathfrak a:\G\widehat\otimes_p\,\overline\G\rightarrow\H\widehat\otimes\,\overline\H$
(injective and with dense image).
\end{proof}

By a slight abuse of notation, we are going to treat $\G$ as a dense subspace of $\H$ and $\G\widehat\otimes_p\,\overline\G$
as a dense subspace of $\H\widehat\otimes\,\overline\H$.
Let us set $\mathscr G(\Si):=\Phi\!\[\G\widehat\otimes_p\,\overline\G\]\subset\mathscr B_2(\Si)\subset\L^2(\Si)$.

\begin{theorem}\label{anghila}
With the algebraic and topological structure induced from $\G\widehat{\otimes}_p\,\overline\G$
and with the scalar product $\<\cdot,\cdot\>_{(\Si)}$, the space $\mathscr G(\Si)$
becomes a Fr\'echet-Hilbert algebra 
(as defined at the end of the Introduction) 
composed of trace-class symbols, continuously and densely embedded
into $\mathscr B_2(\Si)$. 
Its dual $\,\mathscr G'(\Si)$ contains all the operator-bounded symbols.
The restriction $\Phi:\G\widehat\otimes_p\,\overline\G\rightarrow\mathscr G(\Si)$ is an isomorphism of Fr\'echet-Hilbert algebras.
\end{theorem}

\begin{proof}
Clearly $\mathscr G(\Si)$ is turned into a Fr\'echet space by transport of structure.
It is densely contained in $\mathscr B_2(\Si)$ because $\G\widehat\otimes_p\,\overline\G$ is densely contained in
$\H\widehat\otimes\,\overline\H$.

The basic complete tensor products in the case of Hilbert spaces are described in \cite[Sect.48]{Tr}.
For instance, by \cite[Th. 48.3]{Tr}, $\H\widehat\otimes_p\overline\H$
can be identified to $\mathbb B_1(\H)$ while $\H\widehat\otimes_i\,\overline\H$ is canonically
isomorphic to $\mathbb B_\infty(\H)$; we recall that the
Hilbert space tensor product $\H\widehat\otimes\,\overline\H$ is isomorphic to $\mathbb B_2(\H)$.
Taking into account the continuous embedding of $\G\widehat\otimes_p\overline\G$ into $\H\widehat\otimes_p\,\overline\H$\,,
the assertions $\mathscr G(\Si)\subset\mathscr B_1(\Si)$ and then $\mathscr B_1(\Si)'\subset\mathscr G'(\Si)$ become obvious.
But the dual of $\mathbb B_1(\H)$ is $\mathbb B(\H)$\,, which permits the identification of $\mathscr B_1(\Si)'$
with $\mathscr B(\Si)\subset\mathscr G'(\Si)$.

By the very definition of the structure of $\mathscr G(\Si)$ by transport of structure 
from $\G\widehat\otimes_p\,\overline\G$ via $\Phi$, it is enough to check that
$\G\widehat\otimes_p\,\overline\G$ is a Fr\'echet-Hilbert algebra. 
(Recall from the definition made at the end of the Introduction that a Hilbert algebra need not be complete with respect to the topology defined by its scalar product.)
On $\G\otimes\,\overline\G$ the algebraic structure is uniquely defined by $(u\otimes v)^*:=v\otimes u$ and
$(u\otimes v)\cdot(u'\otimes v')=\<u',v\>(u\otimes v')\in\G\otimes\overline\G$\,, valid for $u,v,u',v'\in\G$.
Then one could conclude by density if it is checked that the involution and the multiplication are continuous
with respect to the projective topology.
For the involution this is very easy; we will treat the multiplication.

Let $\{p_\lambda\!\mid\!\lambda\in\Lambda\}$ a directed family of seminorms defining the topology of $\G$.
Since $\G$ is continuously 
contained in $\H$, by \cite[Prop. 7.7]{Tr} there exists $\lambda_0\in\Lambda$ and $C>0$ such that
$\parallel\!u\!\parallel\,\le Cp_{\lambda_0}(u)$ for all $u\in\G$.
Also by \cite[Prop. 43.1]{Tr}, the projective topology on $\G\otimes\,\overline\G$ is defined by the family of seminorms
$\{p_{\lambda,\mu}\mid (\lambda,\mu)\in\Lambda\times\Lambda\}$ given by
\begin{equation*}
p_{\lambda,\mu}(\mathfrak w):=\inf\left\{\sum_l p_\lambda(w_l)p_\mu(w'_l)\ \Big|\ \mathfrak w=\sum_l w_l\otimes w'_l\right\};
\end{equation*}
all the sums should be finite. 
For any $\lambda,\mu\in\Lambda$ and any $\mathfrak u,\mathfrak v\in\G\!\otimes\overline\G$ one has
\allowdisplaybreaks
\begin{align}
p_{\lambda,\mu}(\mathfrak u\cdot \mathfrak v)
=&\inf\Bigl\{\sum_l p_\lambda(w_l)p_\mu(w'_l)\
\Big|\ \mathfrak u\cdot \mathfrak v=\sum_l w_l\otimes w'_l\Bigr\} \nonumber\\
\le&\inf\Bigl\{\sum_{j,k} p_\lambda(\langle v_k,u'_j\rangle u_j)p_\mu(v'_k)\ \Big|\ \mathfrak u=
\sum_j u_j\otimes u'_j\,,\,\mathfrak v=\sum_k v_k\otimes v'_k\Bigr\} \nonumber\\
\le&\inf\Bigl\{\sum_{j,k} \p u'_j\p\,\p v_k\p p_\lambda(u_j)p_\mu(v'_k)\ \Big|\ \mathfrak u=
\sum_j u_j\otimes u'_j\,,\,\mathfrak v=\sum_k v_k\otimes v'_k\Bigr\} \nonumber\\
\le& C^2\inf\Bigl\{\sum_{j} p_\lambda(u_j) p_{\lambda_0}(u'_j)\sum_{k}p_{\lambda_0}(v_k)p_\mu(v'_k) \nonumber\\ 
  & \qquad\qquad \Big\vert  
\mathfrak u=\sum_j u_j\otimes u'_j,\,\mathfrak v=\sum_k v_k\otimes v'_k\Bigr\} \nonumber\\
=& C^2\inf\Bigl\{\sum_{j} p_\lambda(u_j) p_{\lambda_0}(u'_j)\ \Big|\ \mathfrak u=\sum_j u_j\otimes u'_j\Bigr\} 
 \nonumber\\
&\times\inf\Bigl\{\sum_{k}p_{\lambda_0}(v_k)p_\mu(v'_k)\ \Big|\ \mathfrak v=\sum_k v_k\otimes v'_k\Bigr\} \nonumber\\
=& C^2 p_{\lambda,\lambda_0}(\mathfrak u)\,p_{\lambda_0,\mu}(\mathfrak v), \nonumber
\end{align}
which justifies the continuity of the product.
\end{proof}

For further reference we indicate here the continuous embeddings
\begin{equation}\label{om}
\mathscr G(\Si)\hookrightarrow\mathscr B_1(\Si)\hookrightarrow\mathscr B_2(\Si)\hookrightarrow
\mathscr B_\infty(\Si)\hookrightarrow\mathscr B(\Si)\hookrightarrow\mathscr G'(\Si).
\end{equation}
The first and the last one become isomorphisms iff $\G=\H$.
In most of the cases the operators $\pi(s)=\Pi(e_s)^*$ are unitary; in such cases $e_s$ is not a compact symbol.

\begin{remark}\label{limbric}
\normalfont 
The above direct construction of the Fr\'echet-Hilbert algebra $\mathscr G(\Si)$ 
by transport of structure is convenient, because it is
universal, but the output is rather implicit (although, clearly, $\{\phi_{u,v}\mid u,v\in\mathcal G\}$ is a total family in
$\mathscr G(\Si)$\,; see \cite[Th.45.1]{Tr} for a stronger result). Fortunately, in most of the interesting examples,
the space $\mathscr G(\Si)$ has some independent definition as a space of functions (or distributions) on $\Si$.
The scale of spaces given in \eqref{om} can be used to extend the composition law by duality techniques and to define 
optimal "Moyal-type" $^*$-algebras, as in \cite{MP6}.

\end{remark}

Now we need some notions concerning duality of Fr\'echet spaces \cite[Sect.19]{Tr}.
When on the (topological) dual $\G'$ we consider the weak$^*$-topology,
we write $\G'_\sigma$. Recall that in this topology the convergence is just pointwise convergence of functionals and that a base
of neighborhoods of $0\in\G'_\si$ is composed of the polars of all the finite subsets of $\G$.
But we are also going to use the stronger topology $\gamma$ of uniform convergence on convex compact subsets of $\G$,
and then $\G'$ will be denoted by $\G'_\gamma$. One can take as a base of $0\in\G'_\gamma$ the polars of the
convex compact subsets of $\G$.

Using the transpose $\alpha':\H'\rightarrow\G'$ (cf. \cite[Sect.23]{Tr}) and the Riesz antilinear identification
of $\H$ with its strong dual $\H'$, one gets an injective continuous antilinear embedding of $\H$
into the dual $\G'$ (or, equivalently, a linear embedding of $\overline\H$
in $\G'$). We identify thus $\H$ with a subspace of $\G'$, which is dense if on $\G'$ one considers any of the topologies
$\sigma$ or $\gamma$. 
Hence we have a Gelfand triple $(\G,\H,\G'_\nu)$ for $\nu=\si,\gamma$.
Since the duality between $\G$ and $\G'$ is compatible with the scalar product, we can use for
this duality notations as $\<u,w\>:=w(u)$\,, antilinear in $w\in\G'$ and linear in $u\in\G$.

Note that $\G$ can be seen both as the dual of its weak$^*$-dual $\G'_\sigma$ and as the dual of $\G'_\gamma$, cf \cite{Tr}.
In general it does not coincide with the dual of $\G'_\beta$, involving the strong topology $\beta$ of uniform convergence
on the bounded subsets of $\G$. By a simple duality argument it follows that $\H$ (hence $\G$ also)
will be dense in $\G'_\nu$ for $\nu=\sigma,\gamma$. This also happens for $\nu=\beta$ if $\G$ is assumed reflexive.
If $\G$ is a Banach space, we have a Banach-Gelfand triple.
In such a case, the strong topology $\beta$ on $\G'$ coincide with the norm topology given by
$\parallel\!w\!\parallel_{\G'}:=\sup\{\,|\<u,w\>|\,\mid u\in\G,\,\parallel\!u\!\parallel_\G\,\le 1\}$.

The same construction can be applied to the continuous and dense embedding
$\G\widehat\otimes_p\,\overline\G\hookrightarrow\H\widehat\otimes\,\overline\H\,$, getting an ampler Gelfand triple 
$(\G\widehat\otimes_p\,\overline\G,\H\widehat\otimes\,\overline\H,(\G\widehat\otimes_p\,\overline\G)'_\nu)$\,;
one uses the transpose
$\H\widehat\otimes\,\overline\H\overset{\mathfrak a'}{\hookrightarrow}(\G\widehat\otimes_p\,\overline\G)'_\nu$
of the map $\mathfrak a$ constructed in the proof of Lemma \ref{tapir}.
Here, once again, one can use any of the topologies $\nu=\sigma,\gamma,\beta$.

We recall now that we have an isomorphism of $H^*$-algebras $\Phi:\H\widehat\otimes\,\overline\H\rightarrow\mathscr B_2(\Si)$
(in particular a unitary map)
which restricts to an isomorphism of Fr\'echet-Hilbert algebras $\Phi:\G\widehat\otimes_p\,\overline\G\rightarrow\mathscr G(\Si)$.
The inverse of the transpose will be a continuous extension (denoted by abuse by the same letter)
\begin{equation}\label{ciocanitoare}
\Phi:\(\G\widehat\otimes_p\,\overline\G\)'_\nu\rightarrow\mathscr G'(\Si)_\nu\,;
\end{equation}
the topological dual of $\mathscr G(\Si)$ has been denoted by $\mathscr G'(\Si)\,$.

We summarize the discussion above as a Corollary. By {\it isomorphism of Hilbert algebra Gelfand triples} we
denote an isomorphism of Gelfand triples (unitary at the level of Hilbert spaces) for which the "small" spaces are
Fr\'echet-Hilbert algebras, the Hilbert spaces are $H^*$-algebras and the isomorphism respects the $^*$-algebras structures
whenever this makes sense.

\begin{corollary}\label{becata}
Assume that the Fr\'echet space $\G$ is continuously and densely embedded in $\H\,$.
There is a canonical isomorphism of Hilbert algebra Gelfand triples
\begin{equation}\label{cotofana}
\Phi^\pi\equiv\Phi:\(\G\widehat\otimes_p\,\overline\G,\H\widehat\otimes\,\overline\H,(\G\widehat\otimes_p\,\overline\G)'_\nu\)
\rightarrow\(\mathscr G(\Si),\mathscr B_2(\Si),\mathscr G'(\Si)_\nu\).
\end{equation}
\end{corollary}

Notice that for $u,v\in\G'$ one has a well-defined element $\phi_{u,v}:=\Phi(u\otimes v)\in\mathscr G'(\Si)$.

\begin{remark}\label{vrabie}
{\rm One has a canonical isomorphism $(\G\widehat\otimes_p\,\overline\G)'\sim\mathbb B(\G,\G'_\sigma)$.
It is purely algebraical and it involves the space of all the linear operators $A:\G\rightarrow\G'$
which are continuous when on $\G'$ we put the weak$^*$-topology. 
See \cite[pag.465]{Tr} for more details.
Using this, it is easy to deduce from the considerations above
that $\Pi=\Lambda\circ\Phi^{-1}:\mathscr B_2(\Si)\rightarrow
\mathbb B_2(\H)$ extends to a linear isomorphism $\Pi:\mathscr G'(\Si)\rightarrow\mathbb B(\G,\G'_\sigma)\,$.
Thus {\it the elements of $\,\mathscr G'(\Si)$ can be seen as symbols of linear operators $\,T:\G\rightarrow\G'$
that are continuous with respect to the weak$^*$-topology on the dual}. The relation
\begin{equation}\label{libelula}
\<\Pi(g)u,v\>=\int_{\Si}g(t)\<\pi(t)^*u,v\>d\mu(t)=\<g,\phi_{v,u}\>_{(\Si)},
\end{equation}
valid a priori for $g\in\mathscr B_2(\Si)$ and $u,v\in\H$, stands also true for $g\in\mathscr G'(\Si)$ and $u,v\in\G$
with the obvious reinterpretation of the duality $\<\cdot,\cdot\>_{(\Si)}$.
}
\end{remark}

\begin{remark}\label{cocor}
{\rm Some extra structure is present if, in addition, $\G$ satisfies {\it the approximation property}.
That property is shared by many 
specific examples of Fr\'echet spaces; see for instance \cite[Sect.18]{Ja}. 
For us, it serves to identify the injective
tensor product $\G'\widehat\otimes_i\,\overline\G'$ with another topological tensor product $\G'\varepsilon\,\overline\G'$
and thus to simplify the picture. Anyhow, under this extra assumption, one has isomorphisms
\begin{equation*}
\mathbb B(\G,\G'_\sigma)\sim(\G\widehat\otimes_p\,\overline\G)'_\gamma
\cong\G'_\gamma\widehat\otimes_i\,\overline\G'_\gamma.
\end{equation*}
The second one \cite[p. 346]{Ja} is an isomorphism of locally convex spaces and it involves the topology of uniform convergence on compact
convex sets on the various dual spaces. 
By general principles, it justifies the Hilbert algebra Gelfand triple
$(\G\widehat\otimes_p\,\overline\G,\H\widehat\otimes\,\overline\H,\G'_\gamma\widehat\otimes_i\,\overline\G'_\gamma)$.
}
\end{remark}

\begin{remark}
\normalfont 
Let us show how suitable subspaces of $\mathscr B_2(\Si)$ can be used to define the Fr\'echet spaces $\G$ and $\mathscr G(\Si)$ which are the topic of this section.

Let $\mathfrak G(\Si)$ be a Fr\'echet space continuously and densely embedded in $\mathscr B_2(\Si)$. We fix a unit vector $w\in\H$ and we set $w(s):=\pi(s)^*w$ for every $s\in\Si$ and $\phi_w(u):=\phi_{u,w}$ for every $u\in\H$. Then one defines 
\begin{equation*}
\G_w^{\mathfrak G(\Si)}\equiv\G:=\phi_w^{-1}[\mathfrak G(\Si)]\subset\H.
\end{equation*}
It can be shown that $\G$ is a Fr\'echet space continuously and densely embedded in $\H$. 

We are going to give a concrete realization for $\mathscr G(\Si):=\Phi(\G\widehat\otimes_{p}\overline\G)$.
Let us also define $\Upsilon_w:\mathscr B_2(\Si)\rightarrow\L^2(\Si\times\Si)$ by
\begin{equation*}
[\Upsilon_w(g)](s,t):=\<g,\phi_{w(t),w(s)}\>_{(\Si)}=\int_\Si g(r)\<\pi(r)\pi(t)w,\pi(s)w\>d\mu(r).
\end{equation*}
It comes out that $\Upsilon_w$ is an isometry with range contained in 
$\mathscr B_2(\Si)\widehat\otimes\overline{\mathscr B_2(\Si)}$.
Then it is rather easy to show that
\begin{equation*}
\mathscr G(\Si)=\Upsilon_w^{-1}[\mathfrak G(\Si)\widehat\otimes_p\overline{\mathfrak G(\Si)}].
\end{equation*}
The proof is based on the identity $\Upsilon_w\circ\Phi=\phi_w\otimes\phi_w$, which is an immediate consequence of \eqref{ratoi}.
\end{remark}

\section{The Berezin-Toeplitz quantization}\label{brabusca}

In this section we show that some very basic aspects of the Berezin-Toeplitz quantization 
can be naturally recovered in our abstract framework. 
More details on this circle of ideas can be found for instance in \cite{En}. 

Let us fix a unit vector $w\in\H$ and consider the family 
$$W:=\{w(s):=\pi(s)^*w\mid s\in\Si\}.$$ 
As a consequence of 
the existence of the isometry~\ref{cal}, we have in weak sense
\begin{equation}\label{yak}
1=\int_\Sigma\lambda_{w(s),w(s)}
d\mu(s)
\end{equation}
by using the notation introduced in \eqref{Lambda_def}. 
For later use we note that the above equation implies 
\begin{equation}\label{Parseval}
\Tr T=\int_\Sigma \langle Tw(s),w(s)\rangle d\mu(s)
\end{equation} 
for every operator $T\in\BB_1(\Hc)$. 
In fact, this follows by writing $T$ as a linear combination of four nonnegative trace-class operators, 
then using the diagonalization of these four operators along with \eqref{SQ_def}. 
We set
\begin{equation*}
\phi_{w}:\H\rightarrow \mathscr B_2(\Si)\subset\L^2(\Si),\quad\ \[\phi_{w}(u)\]\!(s):=\<u,w(s)\>
\end{equation*}
whose adjoint operator $\phi_{w}^\dagger:\L^2(\Si)\rightarrow\H$ is given by
\begin{equation*}
\phi_{w}^\dagger(f)=\int_\Si\! f(s)w(s)d\mu(s)=\Pi(f)w.
\end{equation*}
The associated kernel is the function $p_w:\Si\times\Si\rightarrow\C$ given by
\begin{equation*}
p_w(s,t):=\<w(t),w(s)\>=\[\phi_w(w(t))\](s)=\overline{\[\phi_w(w(s))\](t)}\,,
\end{equation*}
defining a self-adjoint integral operator $P_w=\Int(p_w)$ in $\L^2(\Si)$.
One checks easily that $P_w=\phi_{w}\phi_{w}^\dagger$ is the final projection of the isometry $\phi_{w}$, so $P_w\!\[\L^2(\Si)\]$ is a closed subspace of $\mathscr B_2(\Si)$. 
Since $\phi_{w}^\dagger\phi_{w}=1$, one has the inversion formula
\begin{equation}\label{bufal}
u=\int_\Si\!\[\phi_{w}(u)\](t)\,w(t)d\mu(t),
\end{equation}
leading to the reproducing formula $\phi_{w}(u)=P_w\[\phi_{w}(u)\]$\,, i.e.
\begin{equation*}
\[\phi_{w}(u)\]\!(s)=\int_\Si\!\<w(t),w(s)\>\[\phi_{w}(u)\]\!(t)d\mu(t).
\end{equation*}
Thus $\mathscr P_w(\Si):=P_w\!\[\L^2(\Si)\]$ is a reproducing Hilbert space with reproducing kernel~$p_w$. 

If in addition  $\sup\Vert\pi(\cdot)\Vert<\infty$ 
and $\Sigma$ is endowed with a topology 
for which $\pi$ is weakly continuous, 
then the reproducing kernel Hilbert space $\mathscr P_w(\Si)$ consists of
bounded continuous functions on~$\Si$.

\medskip
Let us define for $f\in\L^\infty(\Si)$
\begin{equation*}
\Omega^\pi_w(f)\equiv\Omega_w(f):=\int_{\Si}\!f(s)\lambda_{w(s),w(s)}d\mu(s)
\end{equation*}
and call it {\it the Berezin operator associated to the frame $W$}. 
This should be taken in weak sense, i.e. for any $u,v\in\H$ one sets
\begin{equation}\label{viespe}
\<\Omega_w(f)u,v\>:=\int_{\Si}\!f(s)\<\lambda_{w(s),w(s)}u,v\>d\mu(s)
=\int_{\Si}\!f(s)\[\phi_{w}(u)\]\!(s)\overline{\[\phi_{w}(v)\]\!(s)}d\mu(s).
\end{equation}

We gather the basic properties of $\Omega_w$ in the following statement; 
see \cite{Ar,En} for other results of this type. 

\begin{proposition}\label{bondar}
\begin{enumerate}
\item\label{bondar_item1}
The estimate $\p\!\Omega_w(f)\!\p_{\mathbb B(\H)}\,\le\,\p\!f\!\p_{\L^\infty(\Si)}$ holds.
\item\label{bondar_item2}
The map $\Omega_w$ sends $\mu$-a.e. positive functions in positive operators.
\item\label{bondar_item3}
If 
$f\in\L^1(\Sigma,\Vert w(\cdot)\Vert^2\mu)$, then $\Omega_w(f)$ is a trace-class operator and
\begin{equation*}
\Tr\[\Omega_w(f)\]=\int_\Si\!\<\Omega_w(f)w(s),w(s)\>d\mu(s)=\int_\Si\!f(s)\!\p\!w(s)\!\p^2\!d\mu(s).
\end{equation*}
\item\label{bondar_item4}
Assume that $\Si$ is a locally compact space and that $\mu$ is a Radon measure with full support. If $f\in C_0(\Si)$ then $\Omega_w(f)$ is a compact operator.
\end{enumerate}
\end{proposition}

\begin{proof}
1. We estimate
$$
\begin{aligned}
|\<\Omega_w(f)u,v\>|&\le\,\p\!f\!\p_{\L\!^\infty}\!\!\int_\Si\!|\[\phi_{w}(u)\]\!(s)|\,|\[\phi_{w}(v)\]\!(s)|d\mu(s)\\
&\le\,\p\!f\!\p_{\L\!^\infty}\,\p\!\phi_{w}(u)\!\p_{\L^2}\,\p\!\phi_{w}(v)\!\p_{\L^2}\\
&=\,\p\!f\!\p_{\L\!^\infty}\,\p\!u\!\p\,\p\!v\!\p
\end{aligned}
$$
and this gives the result.

\medskip
2. This follows from the fact that $\<\Omega_w(f)u,u\>:=\int_{\Si}f(s)|\[\phi_{w}(u)\]\!(s)|^2d\mu(s)$.

\medskip
3. We may assume $f\ge0$. 
On one hand, if $\{v_k\}_k$ is any orthonormal basis in $\H$, 
one has by the Parseval equality
$$
\begin{aligned}
\Tr\[\Omega_w(f)\]&=\sum_k\<\Omega_w(f)v_k,v_k\>\\
&=\int_\Si\!f(s)\sum_k|\<w(s),v_k\>|^2d\mu(s)\\
&=\int_\Si\!f(s)\!\p\!w(s)\!\p^2\!d\mu(s)
\end{aligned}
$$
hence the assumption $f\in\L^1(\Sigma,\Vert w(\cdot)\Vert^2\mu)$ implies $\Omega_w(f)\in\BB_1(\Hc)$. 
The asserted formula of the trace can then be obtained either by \eqref{Parseval} 
or directly, by \eqref{yak}
\allowdisplaybreaks
\begin{align}
\int_\Si\!\<\Omega_w(f)w(s),w(s)\>d\mu(s)&=\int_\Si \!\int_\Si f(t)\,|\<w(t),w(s)\>|^2d\mu(s)d\mu(t) \nonumber\\
&=\int_\Si\!f(t)d\mu(t)\!\int_\Si \,|\<w(t),w(s)\>|^2d\mu(s) \nonumber\\
&=\int_\Si\!f(t)\p\!w(t)\!\p^2d\mu(t). \nonumber
\end{align}

4. Since $\mu$ was supposed a Radon measure, it is finite on compact subsets of $\Si$ and thus 
$C_{\rm c}(\Si)\subset\L^1(\Si)\cap\L^\infty(\Si)$. If $f$ is continuous and has compact support, 
then $\Omega_w(f)$ is a compact operator by~3. 
 Then the assertion follows by density from~1.
\end{proof}

\begin{remark}\label{forfecuta}
\normalfont 
Formula (\ref{viespe}), which can also be written
\begin{equation*}
\<\Omega_w(f)u,v\>=\<f,\overline{\phi_w(u)}\phi_w(v)\>_{\!(\Si)},
\end{equation*}
opens the way to various extensions by duality. 
Coming back to the setting of Section \ref{graur_sect}, let us also assume that $\mathscr G(\Si)$ is stable under the pointwise product.
Then, for $w,u,v\in\G$, one has $\overline{\phi_w(u)}\cdot\phi_w(v)\in\mathscr G(\Si)\cdot\mathscr G(\Si)\subset\mathscr G(\Si)$. 
This gives meaning to $\Omega^\pi_w(f)$ as a linear continuous operator $\G\rightarrow\G'_\sigma$ for $f\in\mathscr G'(\Si)$.
\end{remark}

We now give a Toeplitz-like form of the operator $\Delta_w(f):=\phi_w\circ\Omega_w(f)\circ\phi_w^{\dag}$.

\begin{proposition}\label{carabus}
For every $f\in \L^\infty(\Sigma)$ one has
\begin{equation*}
\Delta_w(f)=P_w\circ M_f\circ P_w\,,
\end{equation*}
where $M_f\colon \L^2(\Sigma)\to\L^2(\Sigma)$ is the multiplication-by-$f$ operator.
\end{proposition}

\begin{proof}
One has
\allowdisplaybreaks
\begin{align}
\[\Delta_w(f)h\]\!(s)&=\<\Omega_w(f)\[\phi_w^\dag(h)\],w(s)\> \nonumber\\
&=\int_\Si\!f(t)\[\phi_w(\phi_w^\dag(h))\]\!(t)\,\overline{\[\phi_w(w(s))\]\!(t)}\,d\mu(t) \nonumber\\
&=\int_\Si\!f(t)d\mu(t)\!\int_{\Si}\!h(t')\[\phi_w(w(t'))\]\!(t)\,\overline{\[\phi_w(w(s))\]\!(t)}\,d\mu(t') \nonumber\\
&=\int_\Si\int_{\Si}\!f(t)h(t')p_w(t,t')\overline{p_w(t,s)}\,d\mu(t)d\mu(t') \nonumber\\
&=\int_\Si\[\int_{\Si}\!p_w(s,t)f(t)p_w(t,t')d\mu(t)\]\!h(t')d\mu(t') \nonumber\\
&=\[\(P_w\circ M_f\circ P_w\)h\]\!(s). \nonumber
\end{align}
\end{proof}

We define {\it the covariant symbol of the operator} $A\in\mathbb B\!\[\L^2(\Si)\]$
to be the complex function on $\Si$ given by
\begin{equation*}
\[\sigma_w(A)\]\!(s):=\<A\,\phi_w[w(s)],\phi_w[w(s)]\>_{(\Si)}=\<\phi_w^\dag A\,\phi_w[w(s)],w(s)\>.
\end{equation*}
It is equally justified to introduce {\it the covariant symbol of the operator} $S\in\mathbb B(\Hc)$ as
\begin{equation*}
\[\tau_w(S)\](s):=\<Sw(s),w(s)\>.
\end{equation*}
Notice that for $S=\Pi(f)$ one has $\[\tau_w(S)\]\!(s)=\<\Pi(e_s^\star\star f\star e_s)w,w\>$ for any $s\in\Si$.
On the other hand, the covariant symbols of Berezin-Toeplitz operators is also easy to compute
in terms of the reproducing kernel $p_w$. 
Specifically, 
\begin{equation*}
\[\sigma_w(\Delta_w(g))\](s)=\[\tau_w(\Omega_w(g))\](s)=\int_{\Si}\!g(t)|\<w(s),w(t)\>|^2d\mu(t).
\end{equation*}

We now investigate the connection between $\Omega_w\equiv\Omega^\pi_w$ and $\Pi$.

\begin{proposition}\label{cosas}
For any 
$f\in\L^1(\Sigma,\Vert w(\cdot)\Vert^2\mu)$
one has $\Omega_w^\pi(f)=\Pi\[\mathfrak S_w^\pi(f)\]$, where
for $\mu$-almost every $s\in\Si$
\begin{equation*}
\[\mathfrak S_w^\pi(f)\]\!(s):=\int_\Si\!f(t)\<\pi(s)^*w(t),w(t)\>d\mu(t)=\<f,\tau^\pi_w[\pi(s)]\>_{(\Si)}.
\end{equation*}
\end{proposition}

\begin{proof}
Using \eqref{crapcean}, \eqref{Parseval}, \eqref{viespe} and \eqref{yak}, one has
\allowdisplaybreaks
\begin{align}
\Pi^{-1}\[\Omega_w^\pi(f)\](s)&=\Tr\[\pi(s)\Omega_w^\pi(f)\] \nonumber\\
&=\int_\Si\!\<\Omega_w^\pi(f)w(t),\pi(s)^* w(t)\>d\mu(t) \nonumber\\
&=\int_\Si\int_\Si\!f(r)\,\overline{\phi_{w(t),w}(r)}\,\phi_{\pi(s)^*w(t),w}(r)d\mu(t)d\mu(r) \nonumber\\
&=\int_\Si\int_\Si\!f(r)\<\pi(r)^*w,w(t)\>\<w(t),\pi(s)\pi(r)^*w\>d\mu(r)d\mu(t) \nonumber\\
&=\int_\Si\!f(r)\<\pi(r)^*w,\pi(s)\pi(r)^*w\>d\mu(r) \nonumber\\
&=\int_\Si\!f(t)\<\pi(s)^*w(t),w(t)\>d\mu(t) \nonumber \\
&=\[\mathfrak S_w^\pi(f)\]\!(s) \nonumber
\end{align}
and this concludes the proof.  
\end{proof}

\section{Infinite tensor products of square-integrable maps}
\label{next}

We will need the following definition 
in order to describe the square-integrability properties of infinite tensor products of operator-valued maps. 

\begin{definition}\label{app_SQ}
\normalfont
Let $\Hc$ be any complex Hilbert space and $(\Sigma,\Mc)$ be any measurable space as above. 
Also let $\boldsymbol{\mu}=\{\mu_\alpha\}_{\alpha\in A}$ 
be any family of positive measures defined on the $\sigma$-algebra $\Mc$, 
where the index set $A$ is a directed set with respect to some partial ordering. 
A weakly measurable map $\pi:\Si\rightarrow\mathbb B(\H)$ is said to be \emph{square integrable} 
with respect to the family of measures $\boldsymbol{\mu}$ 
if it satisfies the condition 
\begin{equation}\label{app_SQ_def}
\lim_{\alpha\in A}\int_\Sigma\la\pi(s)u_1,v_1\ra \la v_2,\pi(s)u_2\ra d\mu_\alpha(s)
=\la u_1,u_2\ra\la v_2,v_1\ra 
\end{equation}
for all $u_1,u_2,v_1,v_2\in\Hc$.  
\end{definition} 

\begin{remark}
\normalfont
Just as in Remark~\ref{magar}, 
a polarization argument also shows that it suffices to verify \eqref{app_SQ_def} for $u_1=u_2$ and $v_1=v_2$, 
that is 
\begin{equation}\label{app_porc}
(\forall u,v\in\Hc)\quad \lim_{\alpha\in A}\int_\Si\vert\la\pi(s)u,v\ra\vert^2  d\mu_\alpha(s)=
\Vert u\Vert^2\Vert v\Vert^2.
\end{equation}
Also note that if there exists a measure $\mu_\infty$ such that for some $\alpha_0\in A$ we have 
$L^1(\Sigma,\mu_\infty)=\bigcap\limits_{\alpha\ge\alpha_0}L^1(\Sigma,\mu_\alpha)$ 
and for all $\phi\in L^1(\Sigma,\mu_\infty)$ we have 
$\int_\Sigma\phi d\mu_\infty=\lim\limits_{\alpha\in A}\int_\Sigma\phi\de\mu_\alpha$, 
then the operator-valued map $\pi$ is square-integrable with respect to the family 
of measures $\boldsymbol{\mu}=\{\mu_\alpha\}_{\alpha\in A}$ if and only if $\pi\in\SQ(\BB(\Hc),\mu_\infty)$. 
\end{remark}

By using infinite tensor products of Hilbert spaces, 
we introduce the following notion of infinite tensor product of square integrable maps. 
We emphasize that this notion does not involve measure spaces, but rather measurable spaces, 
that is, merely pairs $(\Sigma,\Mc)$ where $\Mc$ is a $\sigma$-algebra of subsets of a set~$\Sigma$. 
In particular, the functions on $\Mc$ must be defined everywhere. 
We use however the above notation from square-integrable operator-valued maps, 
in order to facilitate the application of this infinite tensor product construction in that setting in Theorem~\ref{inf} below. 

\begin{definition}\label{restr}
\normalfont 
For $j\ge 1$, let $\Hc_j$ be any complex Hilbert space,  
$(\Sigma_j,\Mc_j
)$ be  any measurable space, and $\pi_j\colon\Sigma_j\to\BB(\Hc_j)$. 
Assume that we have 
\begin{itemize}
\item a distinguished vector $w_j\in\Hc_j$ with $\Vert w_j\Vert=1$; 
\item a distinguished point $t_j\in\Sigma_j$ with $\pi_j(t_j)=1_{\Hc_j}$. 
\end{itemize}
Denote $\mathbf{w}:=\{w_j\}_{j\ge 1}\in\prod\limits_{j\ge 1}\Hc_j$  and define the complex Hilbert space  $\Hc:=\widehat{\bigotimes\limits_{j\ge 1}}^{\mathbf{w}}\Hc_j$ 
as the inductive limit of the sequence of Hilbert spaces $\Hc^{(N)}:=\mathop{\widehat{\bigotimes}}\limits_{1\le j\le N}\Hc_j$ 
with respect to the isometric embeddings $\Hc^{(N)}\to\Hc^{(N+1)}$, $x\mapsto x\otimes w_{N+1}$. 

Then define 
$$\Sigma:=\Bigl\{s=\{s_j\}_{j\ge 1}\in\prod_{j\ge 1}\Sigma_j\,\Big\vert\,(\exists j(s)\ge 1)(\forall j\ge j(s))\ \  s_j=t_j\Bigr\} $$
and 
$$\pi\colon\Sigma\to\mathbb B(\Hc),\quad \pi(\{s_j\}_{j\ge 1})=\bigotimes_{j\ge 1}\pi_j(s_j).$$
As in \cite[App. D]{Gu72}, we will say that $\Sigma$ is the \emph{restricted Cartesian product} of $\{\Sigma_j\}_{j\ge 1}$ along the sequence of distinguished points $\{t_j\}_{j\ge 1}$, 
and we endow it with the restricted product of the $\sigma$-algebras $\{\Mc_j\}_{j\ge1}$ 
(see \cite[Def. D.2]{Gu72}). 
Moreover, 
$\pi$ is the \emph{restricted tensor product} of the maps $\{\pi_j\}_{j\ge 1}$ along the sequences of distinguished points $\{t_j\}_{j\ge 1}$ 
and distinguished unit vectors $\mathbf{w}=\{w_j\}_{j\ge 1}$. 
\end{definition} 

\begin{remark}\label{restr1}
\normalfont
The definition of an infinite tensor product of operator-valued functions raises several issues which may seem mutually exclusive. 
Namely, in general, the map $\pi$ in Definition~\ref{restr} cannot be defined on the whole Cartesian product 
of $\{\Sigma_j\}_{j\ge 1}$ because of the restrictions required by the definition of an infinite tensor product of operators. 
More precisely, if $T_j\in\mathbb B(\Hc_j)$ for every $j\ge 1$, then, in order to define  
$\bigotimes\limits_{j\ge 1}T_j$ on 
the inductive limit $\widehat{\bigotimes\limits_{j\ge 1}}^{\mathbf{w}}\Hc_j$,  
we need to have $T_jw_j=w_j$ whenever $j\ge 1$ is large enough. 
On the other hand, the problem with the restricted Cartesian product $\Sigma$ is that, 
if some measure $\mu_j$ is given on the measurable space $(\Sigma_j,\Mc_j)$ for every $j\ge1$, 
then it is not clear how to use the sequence of measures $\{\mu_j\}_{j\ge 1}$ in order to define a natural measure on $\Sigma$. 
From this point of view it might seem preferable to work with the full Cartesian product of $\{\Sigma_j\}_{j\ge 1}$ 
and to assume that each $\mu_j$ were a probability measure. 
We will see in a forthcoming paper how these problems can be dealt with 
in some special cases when every $(\Sigma_j,\Mc_j,\mu_j)$ is given by some Gaussian measures on an Euclidean space, 
and every $\pi_j$ is a projective representation of the additive group underlying that Euclidean space. 
\end{remark}

\subsection{Square integrability for infinite tensor products}

In order to avoid the difficulties mentioned in Remark~\ref{restr1}, 
we will use here an alternative approach which leads to Theorem~\ref{inf} below 
and was inspired by some methods used in  
the representation theory of the canonical commutation relations;  
see for instance \cite[Sect. B]{KM65}, \cite[Lemma 4.2]{Re70}, and \cite{Heg69,Heg70}. 
The key of this approach is the approximate orthogonality property 
introduced in Definition~\ref{app_SQ} above. 
The following elementary lemma will be needed in the proof of Theorem~\ref{inf} below. 

\begin{lemma}\label{double}
Let $\{a_{MN}\}_{M,N\ge 1}$ be any double sequence of complex numbers satisfying the conditions: 
\begin{enumerate}
\item There exists $a\in\CC$ for which 
$\lim\limits_{M\to\infty}\lim\limits_{N\to\infty}a_{MN}=a$. 
\item We have $\sup\{\vert a_{MN}\vert\mid M,N\ge 1\}<\infty$. 
\item There exists $\lim\limits_{M\to\infty}a_{MN}=:b_N$ 
 uniformly for $N\ge1$. 
\end{enumerate}
Then $\lim\limits_{N\to\infty}\lim\limits_{M\to\infty}a_{MN}=a$. 
\end{lemma}

\begin{proof}
The sequence $\{b_N\}_{N\ge 1}$ is clearly bounded. 
We need to prove that it is convergent to~$a$, 
and to this end we will prove that every convergent subsequence has the limit~$a$. 

If $\{b_{N_j}\}_{j\ge1}$ is any convergent subsequence, 
then the uniform convergence hypothesis ensures that the double sequence $\{a_{MN_j}\}_{M,j\ge 1}$ 
has the property 
$$a=\lim\limits_{M\to\infty}\lim\limits_{j\to\infty}a_{MN_j}
=\lim\limits_{j\to\infty}\lim\limits_{M\to\infty}a_{MN_j}$$
hence $\lim\limits_{j\to\infty}b_{N_j}=a$ for any  
convergent subsequence $\{b_{N_j}\}_{j\ge1}$, 
and the assertion follows. 
\end{proof}

\begin{theorem}\label{inf}
Assume the setting of Definition~\ref{restr}. 
Moreover assume that there is a $\sigma$-finite measure $\mu_j$ on the measurable space $(\Sigma_j,\Mc_j)$ 
with $\pi_j\in\SQ(\BB(\Hc_j),\mu_j)$ for all $j\ge1$. 
For every $N\ge 1$ 
define 
$$\theta^{(N)}\colon\Sigma_1\times\cdots\times\Sigma_N\to\Sigma,\quad 
\theta^{(N)}(s_1,\dots,s_N)=(s_1,\dots,s_N,t_{N+1},t_{N+2},\dots)$$ 
and consider the measure $\mu^{(N)}:=(\theta^{(N)})_{*}(\mu_1\otimes\cdots\otimes\mu_N)$ on $\Sigma$. 
Then the restricted tensor product $\pi\colon\Sigma\to\BB(\Hc)$ 
of the maps $\{\pi_j\}_{j\ge 1}$ along the sequences of distinguished points $\{t_j\}_{j\ge 1}$ 
and distinguished unit vectors $\mathbf{w}=\{w_j\}_{j\ge 1}$ 
is weakly measurable and square integrable with respect to the sequence of measures $\{\mu^{(N)}\}_{N\ge 1}$. 
\end{theorem}

\begin{proof} 
The proof has two parts, since we will firstly record some preliminary facts. 
It is clear from Definition~\ref{restr} that the map $\pi$ is weakly measurable. 
To prove that this map is square integrable with respect to the sequence of measures $\{\mu^{(N)}\}_{N\ge 1}$ 
we must prove that for arbitrary $ u, v\in\Hc$ we have 
\begin{equation}\label{inf_proof_eq0}
\lim_{N\to\infty}\int_\Sigma\vert\la\pi(\cdot) u, v\ra\vert^2 d\mu^{(N)}=\Vert u\Vert^2\Vert  v\Vert^2
\end{equation}
which is equvalent to 
$$\lim_{N\to\infty}\int_{\Sigma^{(N)}}\vert\la\pi(\theta^{(N)}(s)) u, v\ra\vert^2 d(\mu_1\otimes\cdots\otimes\mu_N)(s)=\Vert u\Vert^2\Vert  v\Vert^2$$
where $\Sigma^{(N)}:=\Sigma_1\times\cdots\times\Sigma_N$ for every $N\ge 1$. 

$1^\circ$  
For every $N\ge 1$ denote by $\Kc^{(N)}$ 
the infinite tensor product of the sequence of Hilbert spaces $\{\Hc_j\}_{j\ge N+1}$ 
along the sequence of unit vectors $\{w_j\}_{j\ge N+1}$.  
The associativity of infinite tensor products shows that there exists a natural unitary operator 
$$W^{(N)}\colon(\Hc_1\widehat{\otimes}\cdots\widehat{\otimes}\Hc_N)\widehat{\otimes}\Kc^{(N)}\to\Hc$$
which also allows us to define 
$$\Hc^{(N)}:=W((\Hc_1\widehat{\otimes}\cdots\widehat{\otimes}\Hc_N)\otimes w_{N+1}\otimes w_{N+2}\otimes\cdots)\subseteq\Hc$$
with the orthogonal projection $P^{(N)}\colon \Hc\to\Hc^{(N)}$. 
Note that 
\begin{equation}\label{star}
M\ge N\ge 1\Longrightarrow P^{(M)}\pi(\theta^{(N)}(\cdot))=\pi(\theta^{(N)}(\cdot))P^{(M)}
\end{equation}
since for every $N\ge 1$ one has the commutative diagram 
\begin{equation}\label{inf_proof_eq1}
\xymatrix{
\Sigma_1\times\cdots\times\Sigma_N \ar[r]^{\hskip1cm \theta^{(N)}} 
\ar[d]_{\pi_1(\cdot)\otimes\cdots\otimes\pi_N(\cdot)\otimes 1_{\Kc^{(N)}}} & \Sigma \ar[d]^{\pi}\\
\BB((\Hc_1\widehat{\otimes}\cdots\widehat{\otimes}\Hc_N)\widehat{\otimes}\Kc^{(N)}) \ar[r] & \BB(\Hc)
}
\end{equation}
whose existence follows by the definition of $\pi$ (see Definition~\ref{restr}) 
and whose bottom arrow is the spatial isomorphism of von Neumann algebras 
given by $T\mapsto W^{(N)}T(W^{(N)})^*$. 
For the same reasons, and by taking into account also Propositions \ref{SQ_prod}~and~\ref{sums}, 
we obtain 
for all $ u_1, v_1, u_2, v_2\in\Hc$ and $M\ge N\ge 1$, 
\allowdisplaybreaks
\begin{align}
\int_\Sigma \vert\la\pi(\cdot) u_1,P^{(M)} v_1\ra & \la P^{(M)} v_2,\pi(\cdot) u_2\ra\vert 
 d\mu^{(N)} \nonumber \\
& \le 
 \Vert P^{(M)} u_1\Vert \Vert P^{(M)} v_1\Vert \Vert P^{(M)} u_2\Vert \Vert P^{(M)} v_2\Vert \nonumber \\
& \label{ssstar}
\le 
 \Vert u_1\Vert \Vert P^{(M)} v_1\Vert \Vert u_2\Vert \Vert P^{(M)} v_2\Vert.
 \end{align}
Note that the above inequality holds under a stronger form if $N\ge M\ge 1$, 
namely 
\begin{equation}\label{sstar} 
\begin{aligned}
N\ge M\Longrightarrow \quad 
\int_\Sigma \la\pi(s) u_1, & P^{(M)} v_1\ra \la P^{(M)} v_2,\pi(s) u_2\ra 
 d\mu^{(N)}(s) \\
& = 
 \la P^{(N)} u_1,P^{(N)} u_2\ra \la P^{(M)} v_2,P^{(M)} v_1\ra
 \end{aligned}
\end{equation}
since in this case we have $P^{(M)}P^{(N)}=P^{(N)}P^{(M)}=P^{(M)}$, 
and then the above equality 
follows by Proposition~\ref{SQ_prod} along with \eqref{inf_proof_eq1}~and~\eqref{star}. 

$2^\circ$. 
We now come back to the proof of \eqref{inf_proof_eq0}. 
By using \eqref{sstar} we obtain 
\begin{equation}\label{inf_proof_eq2}
\lim_{M\to\infty}\lim_{N\to\infty} a_{MN}( u_1, v_1, u_2, v_2)
 =\la u_1, u_2\ra \la v_2, v_1\ra
\end{equation}
where 
$$a_{MN}( u_1, v_1, u_2, v_2):=\int_\Sigma \la\pi(s) u_1, P^{(M)} v_1\ra \la P^{(M)} v_2,\pi(s) u_2\ra 
 d\mu^{(N)}(s)$$
 for all $M,N\ge 1$ and $ u_1, v_1, u_2, v_2\in\Hc$.
 
We will prove that the limits in \eqref{inf_proof_eq2} can be interchanged, 
by using Lemma~\ref{double}. 
To this end it suffices to consider the case $ u_1= u_2=: u$ (by a polarization argument). 
Since 
\begin{equation}\label{inf_proof_eq2.5}
(\forall M,N\ge 1)\quad 
\vert a_{MN}( u, v_1, u, v_2)\vert\le \Vert u\Vert^2\Vert v_1\Vert\Vert v_2\Vert
\end{equation} 
by \eqref{ssstar}, we still need to show that 
there exists a sequence $\{b_N( u, v_1, u, v_2)\}_{N\ge 1}$ 
for which the following conditions are satisfied: 
\begin{equation}
\label{inf_proof_eq3}
\lim_{M\to\infty}a_{MN}( u, v_1, u, v_2)=b_N( u_1, v_1, u_2, v_2)
\text{ uniformly for }N\ge1. 
\end{equation}
In order to check the above condition, 
first note that for arbitrary 
$ u, v\in\Hc$ and $M_1,M_2\ge 1$ we have 
\begin{equation}\label{sssstar}
\int_\Sigma\vert\la\pi(\cdot) u,P^{(M_1)} v\ra 
-\la\pi(\cdot) u,P^{(M_2)} v\ra \vert^2 d\mu^{(N)}
\le \Vert  u\Vert^2\Vert P^{(M_1)} v-P^{(M_2)} v\Vert^2.
\end{equation}
In fact, we may assume $M_2<M_1$. 
If $N<M_1$ and we set $ v_1= v_2:= v-P^{(M_2)} v$, 
then $P^{(M_1)} v_j=P^{(M_1)} v-P^{(M_2)} v$ for $j=1,2$ 
hence by using \eqref{ssstar} for $M:=M_1$ and $ u_1= u_2:= u$ we obtain~\eqref{sssstar}. 
On the other hand, if $N\ge M_1$, then~\eqref{sssstar} follows at once by~\eqref{sstar}.  

Now we can check \eqref{inf_proof_eq3} by proving that 
the $\{a_{MN}( u, v_1, u, v_2)\}_{M\ge 1}$ is a Cauchy sequence, uniformly for $N\ge 1$. 
More precisely, 
for all $M_1,M_2,N\ge 1$ we have 
$$\begin{aligned}
\vert a&_{M_1N}( u, v_1, u, v_2)-a_{M_2N}( u, v_1, u, v_2)\vert \\
\le
&\int_\Sigma\vert \la\pi(\cdot) u,P^{(M_1)} v_1\ra
(\la P^{(M_1)} v_2,\pi(\cdot) u\ra
-
\la P^{(M_2)} v_2,\pi(\cdot) u\ra)\vert d\mu^{(N)} \\ 
&+\int_\Sigma\vert (\la\pi(\cdot) u,P^{(M_1)} v_1\ra
- 
\la\pi(\cdot) u,P^{(M_2)} v_1\ra
\la P^{(M_2)} v_2,\pi(\cdot) u\ra)\vert d\mu^{(N)} \\
\le 
&\Vert u\Vert^2(\Vert P^{(M_1)} v_1\Vert \Vert P^{(M_1)} v_2
- P^{(M_2)} v_2 \Vert
+\Vert P^{(M_1)} v_1\Vert \Vert P^{(M_1)} v_2
- P^{(M_2)} v_2 \Vert)
\end{aligned}$$
where the last inequality follows by the Schwartz inequality along with the estimates 
\eqref{ssstar}--\eqref{sssstar}. 
Thus \eqref{inf_proof_eq3} follows.  

Now \eqref{inf_proof_eq2.5} along with \eqref{inf_proof_eq3} 
ensure that Lemma~\ref{double} applies, 
hence the limits in \eqref{inf_proof_eq2} can be interchanged.  
In turn, this implies that   
\begin{equation}\label{inf_proof_eq4}
(\forall u_1, v_1, u_2, v_2\in\Hc )\quad 
\lim_{N\to\infty}\lim_{M\to\infty} a_{MN}( u_1, v_1, u_2, v_2)
 =\la u_1, u_2\ra \la v_2, v_1\ra. 
\end{equation} 
On the other hand, by taking into account the definition of $a_{MN}( u_1, v_1, u_2, v_2)$, 
it follows that for all $N\ge 1$ and $ u, v\in\Hc$ we have  
$$\lim_{M\to\infty} a_{MN}( u, v, u, v)
=
\lim_{M\to\infty} \int_\Sigma \vert\la\pi(\cdot) u, P^{(M)} v\ra\vert^2  
 d\mu^{(N)}
 =\int_\Sigma\vert\la\pi(\cdot) u, v\ra\vert^2 d\mu^{(N)} 
$$
where the last equality is a direct consequence of \eqref{sssstar}. 
Thus \eqref{inf_proof_eq4} implies that 
 the equality \eqref{inf_proof_eq0} holds true, and this completes the proof. 
\end{proof}

\subsection{Symbol calculus for infinite tensor products} 
 
One can inquire what happens at the level of quantizations when operations (as tensor products) with square integrable families are done. 
This question is particularly interesting in the setting of Section \ref{brabusca} since our input there was 
a pointed Hilbert space $(\H,w)$, which is the basic object when performing infinite tensor products. 
In the rest of this section we will take a few steps in this direction, 
but many interesting problems remain unsolved yet, 
in particular regarding the symbols that give rise to operators in various Schatten ideals; 
see for instance Proposition~\ref{bondar}\eqref{bondar_item3} above.

\begin{theorem}\label{last}
Assume the setting of Theorem~\ref{inf} 
and denote by $\widetilde{\L}^\infty(\Sigma)$ 
the space of all complex-valued bounded measurable functions on~$\Sigma$. 
Then the following assertions hold: 
\begin{enumerate}
\item\label{last_item1}
 For every $N\ge 1$, $f\in \widetilde{\L}^\infty(\Sigma)$  and $ u, v\in\Hc$ the integral   
$$\Omega^{(N)}_u(f) v:=\int_\Sigma f(\cdot) \la v,\pi(\cdot) u\ra\pi(\cdot) u d\mu^{(N)} $$
is weakly convergent and defines an operator 
$\Omega^{(N)}_u(f)\in\BB(\Hc)$ 
satisfying  $\Vert\Omega^{(N)}_u(f)\Vert\le \Vert u\Vert^2\sup\limits_\Sigma\vert f\vert$. 
\item\label{last_item2} 
For all $ u\in\Hc\setminus\{0\}$ and $f\in\widetilde{\L}^\infty(\Sigma)$ 
there exists $\Omega_u(f)\in\BB(\Hc)$ 
satisfying $\Omega_u(f)=\lim\limits_{N\to\infty}\Omega^{(N)}_u(f)$ 
in the weak operator topology, 
and moreover $\Vert\Omega_u(f)\Vert\le\Vert u\Vert^2\sup\limits_\Sigma\vert f\vert$. 
\item\label{last_item3} 
If $ u\in\Hc$ and $0\le f\in\widetilde{\L}^\infty(\Sigma)$, then $0\le \Omega_u(f)\in\BB(\Hc)$. 
\item\label{last_item4} 
For every $ u\in\Hc$ with $\Vert u\Vert=1$ we have $\Omega_u(1)=1$. 
\end{enumerate}
\end{theorem}

\begin{proof}
For Assertion~\eqref{last_item1} we 
use again the notation $\Sigma^{(N)}=\Sigma_1\times\cdots\times\Sigma_N$ to write 
$$\begin{aligned}
\la \Omega^{(N)}_u(f) v_1,& v_2\ra  \\
&=\int_\Sigma f(\cdot)  \la v_1,\pi(\cdot) u\ra\la \pi(\cdot) u, v_2\ra d\mu^{(N)} \\
&=\int_{\Sigma^{(N)}}f(\theta^{(N)}(\cdot)) \la v_1,\pi(\theta^{(N)}(\cdot)) u\ra 
\la\pi(\theta^{(N)}(\cdot)) u, v_2\ra d(\mu_1\otimes\cdots\otimes\mu_N)
\end{aligned}$$
hence, by taking into account the commutative diagram~\eqref{inf_proof_eq1}, 
the convergence of the above integral and the required estimate for the norm of $\Omega^{(N)}_u(f)$
follow by the estimate provided by Proposition~\ref{sums}. 

For Assertion~\eqref{last_item2} 
we only need to prove the asserted convergence in the weak operator topology, 
since the norm estimate will then follow by the norm estimates established above. 

In order to prove that the sequence $\{\Omega^{(N)}_u(f)\}_{N\ge1}$ 
is convergent in the weak operator topology, 
we will adapt the method of proof of Theorem~\ref{inf} 
and we will use the notation from that proof,  
except that for fixed $f\in \widetilde{\L}^\infty(\Sigma)$, 
we set 
$$a_{MN}( u_1, v_1, u_2, v_2):=\int_\Sigma f(\cdot)\la\pi(\cdot) u_1, P^{(M)} v_1\ra \la P^{(M)} v_2,\pi(\cdot) u_2\ra 
 d\mu^{(N)}$$
 for all $M,N\ge 1$ and $ u_1, v_1, u_2, v_2\in\Hc$. 
Then we have 
$$\lim\limits_{M\to\infty}a_{MN}( u, v_1, u, v_2)=\la \Omega^{(N)}_u(f) v_1, v_2\ra 
\text{ uniformly for }N\ge1$$
by a reasoning similar to the one used for proving \eqref{inf_proof_eq3}. 
Moreover, we have the following version of \eqref{inf_proof_eq2.5}
\begin{equation*}
(\forall M,N\ge 1)\quad 
\vert a_{MN}( u, v_1, u, v_2)\vert\le \Vert u\Vert^2\Vert v_1\Vert\Vert v_2\Vert\sup\limits_\Sigma\vert f\vert
\end{equation*} 
hence we can use Lemma~\ref{double} for proving that 
the sequence $\{\la \Omega^{(N)}_u(f) v_1, v_2\ra \}_{N\ge 1}$ is convergent 
for all $ u, v_1, v_2\in\Hc$.  
That is, the operator sequence $\{\Omega^{(N)}_u(f)\}_{N\ge1}$ 
is convergent in the weak operator topology in $\BB(\Hc)$ for all $ u\in\Hc$. 

Assertion~\eqref{last_item3} is clear. 

Finally, Assertion~\eqref{last_item4} follows directly from Theorem~\ref{inf} 
(see \eqref{inf_proof_eq0}),  and this completes the proof. 
 \end{proof}

\begin{remark}
\normalfont
The above Theorem~\ref{last}\eqref{last_item4} should be regarded as a version of \eqref{yak}. 
In the present situation of infinite tensor products 
we do not have any natural version of the space $\L^2(\Sigma)$, 
and therefore we need to work inside of the space $\CC^\Sigma$ of all complex valued functions on~$\Sigma$ 
in order to construct the reproducing kernel Hilbert space as above. 
Namely, if we pick any unit vector $w\in\Hc$, 
then we can define $w(\cdot)=\pi(\cdot)^*w\colon\Sigma\to\Hc$ 
and $\phi_w\colon\Hc\to\CC^\Sigma$, $\phi_w(u)=\la u,w(\cdot)\ra=\la\pi(\cdot)u,w\ra$. 
Then the linear map $\phi_w$ is injective by \eqref{inf_proof_eq0}, 
hence we may define $\mathscr P_w(\Sigma):=\Ran\phi_w$ and make this function space 
into a Hilbert space such that $\phi_w\colon\Hc\to\mathscr P_w(\Sigma)$ be a unitary operator. 
Note that $\mathscr P_w(\Sigma)$ is a reproducing kernel Hilbert space 
since for every $s\in\Sigma$ the point evaluation 
$\ev_s\colon \mathscr P_w(\Sigma)\to\CC$, $\ev_s(f)=f(s)$ is continuous.  
The corresponding reproducing kernel is 
$$
p_w\colon\Sigma\times\Sigma\to\CC,\quad p_w(s,t)=\la w(t),w(s)\ra
$$ 
as above (see \cite[Th. I.1.6]{Ne00}). 

The version of the inversion formula \eqref{bufal} in the present setting 
is 
$$u=\lim_{N\to\infty}\Omega^{(N)}_w(1)u= 
\lim_{N\to\infty}\int_\Sigma \la u,\pi(\cdot)^*w\ra\pi(\cdot)^*w d\mu^{(N)}$$
where the limit and the integral are taken in the weak sense. 
This is obtained by using Theorem~\ref{last}(\eqref{last_item2},\eqref{last_item4}) 
for the mappings $\pi_j(\cdot)^*\in\SQ(\BB(\Hc_j),\mu_j)$ 
and provides a generalization of a result \cite[Th. 3.2]{KM65} from 
the representation theory of canonical commutation relations.  
\end{remark}

\section{Some examples}\label{marlita}

It is quite clear that the formalism of the previous section is meant to cover at least two situations: square integrable irreducible unitary group representations with their associated twisted convolution algebras 
and the Weyl pseudodifferential calculus. 
We will now briefly indicate other examples, in order to show the generality of our setting. 
They are developped just to the extent when the identification of the relevant objects becomes obvious, 
but we plan to give more specific applications in our forthcoming papers. 
The references cited in this section contain much more than we are able to review here.

\subsection{The magnetic Weyl calculus}\label{avat}

One takes $\Si:=\X\times\X^*$, where $\X$ is a $n$-dimensional real vector space and $\X^*$ is its dual (so the ``phase-space'' $\Si$ is non-canonically isomorphic to $\mathbb R^{2n}$). 
Below,  setting $B=0$ and $A=0$, 
one would recover the standard Weyl calculus \cite{Fo}.

The magnetic Weyl calculus \cite{IMP,KO,MP1,MP4,MPR2} has as a background the problem of quantization of a physical system consisting in a spin-less particle moving in the
euclidean space $\X\cong\mathbb R^n$ under the influence of a magnetic field, i.e.
a closed $2$-form $B$ on $\X$ ($dB=0$), given in a base by matrix-component functions
$B_{jk}=-B_{kj}:\X\rightarrow\mathbb R$, $j,k=1,\dots,n$.
For simplicity and in order to have a full formalism we are going to assume that the components $B_{jk}$ belong to $C^\infty_{\rm{pol}}(\X)$,
the class of smooth functions on $\X$ with polynomial bounds on all the derivatives. 
The magnetic field can be written
in many ways as the differential $B=dA$ of some $1$-form $A$ on $\X$ called {\it vector potential}.
One has $B=dA=dA'$ iff $A'=A+d\varphi$ for some $0$-form $\varphi$ (then they are called equivalent).
It is easy to see that vector potential can also be chosen of class $C^\infty_{\rm{pol}}(\X)$.

One would like to develop a symbol calculus taking the magnetic field into account. Basic requirements are:
(i) it should reduce to the standard Weyl calculus for $A=0$ and (ii) the operators $\Pi^A\!(f)$ and $\Pi^{A'}\!(f)$ should
be unitarily equivalent (independently on the symbol $f$) if $A$ and $A'$ are equivalent; this is called {\it gauge covariance}
and has a fundamental physical meaning. To justify the formulae, one could think of the emerging
symbol calculus as a functional calculus for the family of non-commuting self-adjoint operators
$(Q_1,\dots,Q_n;P^A_1,\dots,P^A_n)$ in $\H:=\L^2(\X)$. Here $Q_j$ is one of the components of
the position operator, but the momentum $P_j:=-i\partial_j$ is replaced by {\it the magnetic momentum}
$P^A_j:=P_j-A_j(Q)$ where $A_j(Q)$ indicates the operator of multiplication with the function $A_j\in C^\infty_{\rm{pol}}(\X)$.
Notice the commutation relations
\begin{equation*}
i[Q_j,Q_k]=0,\quad i[P^A_j,Q_k]=\delta_{j,k},\quad i[P^A_j,P^A_k]=B_{jk}(Q).
\end{equation*}
Now one computes {\it the magnetic Weyl system}
\begin{equation*}
\pi^A:\Si\rightarrow\mathbb B(\H),\quad\ \pi^A(x,\xi):=\exp\[i\(x\cdot P^A-Q\cdot\xi\)\]
\end{equation*}
and gets explicitly 
\begin{equation*}
\[\pi^A(x,\xi)u\](y)=e^{-i\(y+\frac{x}{2}\)\cdot\xi}\,\exp\[(-i)\!\underset{[y,y+x]}{\int}\!A\]\,u(y+x).
\end{equation*}
The extra phase factor involves the circulation of the $1$-form $A$ through the segment $[x,y]:=\{(1-t)x+ty\mid t\in[0,1]\}$.
These operators depend strongly continuous of $(x,\xi)$ and satisfy $\pi^A(0,0)=1$ and $\pi^A(x,\xi)^*=\pi(x,\xi)^{-1}=\pi^A(-x,-\xi)$ (thus being unitary). {\it However they do not form a projective representation of $\,\Si=\X\times\X^*$.}  Actually they satisfy
\begin{equation*}
\pi^A(x,\xi)\,\pi^A(y,\eta)=\Omega^B[(x,\xi),(y,\eta);Q]\,\pi^A(x+y,\xi+\eta)\,,
\end{equation*}
where $\Omega^B[(x,\xi),(y,\eta);Q]$ only depends on the $2$-form $B$ and denotes the operator of multiplication
in $\L^2(\X)$ by the function
\begin{equation*}
\X\ni z\rightarrow\Omega^B[(x,\xi),(y,\eta);z]:=\exp\left[\frac{i}{2}\,(y\cdot\xi-x\cdot\eta)\right]
\exp\left[(-i)\!\!\!\!\!\underset{<z,z+x,z+x+y>}{\int}\!\!\!\!B\right].
\end{equation*}

Here the distinguished factor is constructed with the flux (invariant integration) of the magnetic field
through the triangle defined by the corners $z$, $z+x$ and $z+x+y$.

A straightforward computation leads to 
$$
\begin{aligned}
\[\Phi^A(u\otimes v)\](x,\xi):=&\<\pi^A(x,\xi)u,v\>\\
=&\int_\X e^{-iy\cdot\xi}\,\exp\[(-i)\!\!\underset{[y-x/2,y+x/2]}{\int}\!\!A\]u(y+x/2)\,\overline{v(y-x/2)}dy.
\end{aligned}
$$
It can be decomposed into the product of the multiplication by a function with values in the unit circle,
a change of variables with the Jacobian identically equal to~1, and a partial Fourier transform. All are isomorphisms between the corresponding spaces, so the orthogonality relation holds with $\mathscr B_2(\Si)=\L^2(\Si)$. 

Thus one can apply all the prescriptions and get the correspondence $f\mapsto\Pi^A(f)$ and the composition law
$(f,g)\rightarrow f\star^B g$ (depending only on the magnetic field).
In fact people are interested in the (symplectic) Fourier transformed version
$a\(Q,P^A\)\equiv\Op^A(a):=\Pi^{A}[\mathfrak F^{-1}(f)]$ and in the multiplication $\#^B$ obtained by transport of structure and
therefore satisfying $\Op^A(a)\Op^A(b)=\Op^A(a\#^B b)$. The resulting involution is just complex conjugation, thus
$\Op^A(a)^*=\Op^A\(\overline a\)$. For the convenience of the reader we indicate the explicit formulae, in which
we set $\Gamma^A([x,y]):=\int_{[x,y]}A$ and $\Gamma^B(<x,y,z>):=\int_{<x,y,z>}\!B$.
The {\it magnetic Moyal product} is
\begin{equation*}
\begin{aligned}
\left(a\#^B b\right)(X)=\pi^{-2n} \int_\Si\int_\Si & \exp\left\{-2i\[(x-z)\cdot(\xi-\eta)-
(x-y)\cdot(\xi-\zeta)\]\right\} \\
& \times\exp\left[-i\Gamma^B(<x-y+z,y-z+x,z-x+y>)\right] \\
& \times f(Y)g(Z)
dZ\, dY
\end{aligned}
\end{equation*}
and {\it the magnetic Weyl calculus} is given by
\begin{equation}\label{op}
\begin{aligned}
\left[\mathfrak{Op}^{A}(a)u\right](x)=(2\pi)^{-n}\int_\X\int_{\X'}\!
 & \exp\left[i(x-y)\cdot\xi\right]
\exp\left[-i\Gamma^A([x,y])\right] \\ 
&\times a\left(\frac{x+y}{2},\xi\right)u(y)dy\,d\xi.
\end{aligned}
\end{equation}
An important property of (\ref{op}) is {\it gauge covariance}, as hinted above:
if $A'=A+d\rho$  defines the same magnetic field as $A$, then $\Op^{A'}\!(f)=e^{i\rho}\,\Op^{A}(f)\,e^{-i\rho}$.
By killing the magnetic phase factors in all the formulae above one gets the defining relations of the usual Weyl calculus.

A convenient choice of the auxiliary space in this case is the Schwartz space $\G=\Sc(\X)$, 
which is a nuclear Fr\'echet space continuously and densely embedded in $\L^2(\Si)$; 
thus $\G'$ will be the space of tempered distributions. 
By a simple examination of the map~$\Phi^A$, this leads to $\mathscr G(\Si)=\Sc(\X\times\X^*)$. 
It can easily be shown that (by suitable restriction or extensions) $\Op^A[\Sc(\X\times\X^*)]=\mathbb B[\Sc'(\X),\Sc(\X)]$ and $\Op^A[\Sc'(\X\times\X^*)]=\mathbb B[\Sc(\X),\Sc'(\X)]$. 
The symbol algebras for the magnetic Weyl calculus were studied in detail in \cite{MP1}, while in \cite{IMP} the full pseudodifferential theory was developed.

\subsection{Operator calculi on locally compact abelian groups}\label{babusca}

In this subsection we will present some square-integrable operator-valued maps 
related to the metaplectic representation in the framework of 
locally compact abelian groups, 
representation which was studied in \cite{Se63,We64,Ma65}. 

The framework is provided by any locally compact abelian group $(G,+)$ 
with its dual group $\widehat{G}$, 
which is the set of all continuous homomorphisms from $G$ into the circle group $\TT$.  
Recall that $\widehat{G}$ is in turn a locally compact abelian group 
with the pointwise operations and with the topology given by uniform convergence on compact sets. 
The natural duality pairing between $\widehat{G}$ and $G$ is denoted by 
$\langle \cdot,\cdot\rangle\colon \widehat{G}\times G\to\TT$. 
For every Haar measure $\nu$ on $G$ there exists a unique Haar measure $\nu^*$ on $\widehat{G}$ 
for which the Fourier transform 
$$\Fc\colon L^1(G,\nu)\to BC(\widehat{G}),\quad 
(\Fc f)(\xi)=\int_G\langle\xi,-x\rangle f(x) d\nu(x)$$
gives rise to a unitary operator $L^2(G,\nu)\to L^2(\widehat{G},\nu^*)$.

\begin{proposition}\label{meta}
Let $G$ be any locally compact abelian group with a Haar measure~$\nu$ 
and denote $\Hc=L^2(G,\nu)$. 
For $k=1,2$ define 
$$\pi_k\colon G\times\widehat{G}\to \BB(\Hc),\quad 
(\pi_k(x,\xi)u)(z)=\langle\xi,kz+(k-1)x\rangle u(z+x).$$
Then the following assertions hold: 
\begin{enumerate}
\item\label{meta_item1} 
We have $\pi_1\in\SQ(\BB(\Hc),\nu\times\nu^*)$ 
and $\Phi^{\pi_1}\colon\Hc\widehat{\otimes}\overline{\Hc}\to L^2(G\times\widehat{G},\nu\times\nu^*)$ 
is a unitary operator. 
\item\label{meta_item2} 
If the map $x\mapsto 2x$ is an automorphism of $G$, 
then there exists a constant $c>0$ for which $\pi_2\in\SQ(\BB(\Hc),\nu\times c\nu^*)$ 
and the corresponding operator 
$\Phi^{\pi_2}\colon\Hc\widehat{\otimes}\overline{\Hc}\to L^2(G\times\widehat{G},\nu\times c\nu^*)$ 
is unitary. 
\end{enumerate}
\end{proposition}

\begin{proof} 
For $k=1,2$ and 
every $u\in L^1(G\times\widehat{G},\nu\times\nu^*)\cap L^2(G\times\widehat{G},\nu\times\nu^*)$ define 
$$T^{\pi_k}u=\iint_{G\times\widehat{G}} u(x,\xi)\pi_k(x,\xi) d\nu(x) d\nu^*(\xi).$$
We now prove the assertions separately. 

\eqref{meta_item1} 
It follows by \cite[I.8]{We64} that $T^{\pi_1}$ extends to a unitary operator $L^2(G\times\widehat{G},\nu\times\nu^*)\to\BB_2(\Hc)$. 
By taking operator adjoints and complex-conjugates of functions, 
we then obtain a unitary operator $L^2(G\times\widehat{G},\nu\times\nu^*)\to\BB_2(\Hc)$ 
given by 
$$u\mapsto \iint_{G\times\widehat{G}} u(x,\xi)\pi_1(x,\xi)^* d\nu(x) d\nu^*(\xi) $$ 
for every 
$u\in L^1(G\times\widehat{G},\nu\times\nu^*)\cap L^2(G\times\widehat{G},\nu\times\nu^*)$, 
and this is just the adjoint of 
$\Phi^{\pi_1}\circ\Lambda^{-1}\colon\Hc\widehat{\otimes}\overline{\Hc}\to L^2(G\times\widehat{G},\nu\times\nu^*)$, 
Corollary~\ref{adjoint}.  
Since $\Lambda$ is a unitary operator, it follows that $\Phi^{\pi_1}$ is in turn unitary, as asserted. 

\eqref{meta_item2} 
If the map $x\mapsto 2x$ is an automorphism of $G$, 
then it follows by \cite[Th. 1]{Se63} that there exists a constant $c>0$ 
for which $T^{\pi_2}$ extends to a unitary operator $L^2(G\times\widehat{G},\nu\times c\nu^*)\to\BB_2(\Hc)$. 
Now the assertion can be proved just as above. 
\end{proof}

The hypothesis that the map $x\mapsto 2x$ is an automorphism of $G$ from Proposition~\ref{meta}\eqref{meta_item2} 
is satisfied by many important examples of groups, as for instance 
the linear spaces $G=\R^d$
or the additive groups of local fields, like the $p$-adic fields $\mathbb Q_p$ 
which are interesting for quantization and pseudodifferential theory 
with applications to number theory, as shown in 
\cite{Ha}. 

We also note that the aforementioned hypothesis is quite natural 
inasmuch as it is satisfied if and only if 
an appropriate version of the Stone-von Neumann theorem holds. 
More precisely, according to \cite[Th. 1]{Ma65}, 
if we define the cocycle 
$$\sigma\colon(G\times\widehat{G})\times(G\times\widehat{G})\to\TT,
\quad \sigma((x,\xi),(y,\eta))=\langle\xi,y\rangle\overline{\langle\eta,x\rangle},$$
then the locally compact abelian group $G\times\widehat{G}$ 
has just one equivalence class of unitary irreducible projective representations with the cocycle~$\sigma$ 
if and only if the map $x\mapsto 2x$ is an automorphism of $G$. 
One projective representation of that type is just the map $\pi_2$ from our Proposition~\ref{meta}\eqref{meta_item2}.   
See also \cite[App. VIII]{Ne00} for a discussion of this circle of ideas.


A lot of extra structure is present due to the existence of a Fourier transform and to the structure theorem for locally compact abelian groups. 
In particular, besides choosing for the ingredient $\G$ the Bruhat-Schwartz space, 
there also exists a better choice relying on writting $G$ as $\mathbb R^m\times G_0$ with $G_0$ containing an open compact subgroup. 
We do not review the theory, for which we refer to \cite{GS} which also contains a lot of constructions and results involving a certain class of coorbit spaces, discretization techniques and Gabor frames. 
If $G=\R^m$ (i.e., $G_0$ is trivial) the emerging formalism boils down essentially to the Kohn-Nirenberg pseudodifferential calculus \cite{Fo}.

\subsection{Unitary representations of some infinite dimensional Lie groups}\label{lostrita}

Relevant references are \cite{Pe,BB1,BB2,
BB11a}, to which we refer for a full presentation.

The starting point is an unitary strongly continuous representation $\varpi:M\rightarrow\mathbb U(\H)$, where
$M$ is a locally convex  Lie group with the Lie algebra $\mathfrak m$ and a smooth exponential
$\exp_M:\mathfrak m\rightarrow M$. 
On the dual $\mathfrak m'$ of $\mathfrak m$ we consider the weak$^*$ topology.

We also fix a real finite dimensional vector space $\Si$ with dual $\Si'$ and a linear map
$\theta:\Si\rightarrow\mathfrak m$. The basic idea is to use $\pi:=\varpi\circ\exp_M\!\circ\,\theta:\Si\rightarrow\mathbb B(\H)$
as well as the Fourier transform $\,\widehat\cdot:\L^2(\Si)\rightarrow\L^2(\Si')$ in order to build
a very general form of the Weyl calculus. So one should set
\begin{equation}\label{omida}
\Op^{\varpi,\theta}(a)\equiv\Pi\(\widehat a\):=\int_\Si\widehat a(s)\,\varpi\[\exp_M(\theta(s))\]ds.
\end{equation}
The outcome was called in \cite{BB11a} {\it the localized Weyl calculus associated to the representation $\varpi$ along the linear map $\theta$}.
The single requirement needed to develop the basic part of the theory is orthogonality, i.e.
\begin{equation*}
\int_\Si \, |\!\<\varpi\[\exp_M(\theta(s))\]u,v\>\!|^2ds=\,\p\!u\!\p^2\,\p\!v\!\p^2,\ \quad\forall\,u,v\in\H.
\end{equation*}
Under this requirement, the general theory gives a definite sense to (\ref{omida}) at least for $\widehat a\in\mathscr B_2(\Si)$
and the constructions and results of sections \ref{vaca} and \ref{platica} are valid.
A good choice for the auxiliary space $\G$ is the space of smooth vectors of the representation~$\varpi$:
\begin{equation*}
\G\equiv\H_\infty:=\{u\in\H\mid M\ni m\mapsto\varpi(m)u\in\H\ {\rm is}\ C^\infty\}.
\end{equation*}
It carries a natural Fr\'echet topology \cite[Remark 2.1]{BB11a} in terms of "differential operators" indexed
by the universal associative enveloping algebra $U(\mathfrak m_\mathbb C)$ of the complexified Lie algebra
$\mathfrak m_\mathfrak C$. 
In the infinite dimensional case the denseness of $\H_\infty$ in $\H$ is not always
verified so we must require it. 
But as soon as this is achieved, all the results of the present article hold.
In \cite{BB11a} extra regularity assumptions are imposed in order to have good control upon the spaces involved.
One of the aims is to identify $\mathscr G(\Si)$ with the Schwartz space $\mathscr S(\Si)$ and
$\Op^{\varpi,\theta}\[\widehat{\mathscr G(\Si)}\]=\Pi\[\mathscr G(\Si)\]$ with the Fr\'echet
space $\mathbb B(\H)_\infty$ of all the smooth vectors (operators) under the continuous unitary representation
\begin{equation*}
\varpi^{(2)}:M\times M\rightarrow\mathbb B\[\mathbb B_2(\H)\],\
\quad \[\varpi^{(2)}(m,n)\]T:=\varpi(m)T\varpi(n)^{-1}.
\end{equation*}
Another one is to determine when $\mathscr B_2(\Si)=\L^2(\Si)$ holds.

The above general setting was studied in some detail in two specific situations: 
Firstly, when $M$ is a finite-dimensional connected and simply connected nilpotent Lie group, 
a situation which was pointed out in \cite{Pe} and will be discussed in subsection~\ref{Ped} below.  
Secondly, the case when $M$ is an infinite-dimensional Lie group that 
can be written as the semidirect product $\mathcal F\ltimes G$ between a (finite dimensional) connected nilpotent
Lie group $G$, with Lie algebra $\mathfrak g$, and a suitable (typically infinite dimensional) locally convex space $\mathcal F$
of smooth functions on $G$; 
this was studied in in \cite{BB1, BB11a}. 
The connection with the square-integrable families of operators is established 
by \cite[Cor. 4.7(3)]{BB11a}. 

An important particular case comes from the presence of a smooth magnetic field (i.e., a closed differential 2-form) on $G$, 
inasmuch as the aforementioned function space $\Fc$ should be invariant under left translations 
on $G$ and should contain the coefficients of the magnetic field as well as their derivatives of arbitrarily high order. 
This shows that if the magnetic field fails to have polynomial coefficients, then $\Fc$ is infinite-dimensional. 
In this situation an irreducible representation is given by 
$$\varpi\colon M=\Fc\rtimes G\to\BB(L^2(G)),\quad 
(\varpi(\phi,x)f)(y)=e^{i\phi(y)}f(x^{-1}y)$$
and it was proved in \cite{BB1} that although $M$ is an infinite-dimensional Lie group, 
its irreducible representation $\varpi$ can be obtained by the geometric quantization 
from a certain \emph{finite-dimensional} coadjoint orbit $\Oc$ of $M$. 
Moreover, the coadjoint orbit~$\Oc$ is symplectomorphic to the cotangent bundle $T^*G$. 
The appearance of the finite dimensional vector space $\Si$ from the above general framework 
is connected to that coadjoint orbit $\Oc$ and the linear mapping $\theta$ is assigned canonically to a vector potential
generating the magnetic field. 
Specifically, one may use $\Si=\mathfrak g\times\mathfrak g'$. 
The outcome of the operator calculus in this setting 
is an extension to nilpotent Lie groups of the magnetic Weyl calculus
briefly presented in subsection~\ref{avat}, which can be recovered for the Abelian group $G=(\mathbb R^n,+)$.

\subsection{Operator calculus for 
representations of nilpotent Lie groups}\label{Ped}
We will briefly describe some square-integrable operator-valued maps related to 
unitary irreducible representations of nilpotent Lie groups. 
This method was already used in the proof of Corollary~\ref{irred3} above.
The details of this construction can be found essentially in \cite{Pe}; see also \cite{BB2}. 

Let $G$ be any connected, simply connected, nilpotent Lie group with the Lie algebra~$\gg$.  
 Then the exponential map $\exp_G\colon\gg\to G$ is a diffeomorphism 
 with its inverse denoted by $\log_G\colon G\to\gg$. 
The adjoint action of $G$ is 
$$\Ad_G\colon G\times\gg\to\gg,\quad \Ad_G(g)x:=\frac{d}{dt}\Big\vert_{t=0}(g\exp_G(tx)g^{-1}) $$
 We denote by $\gg^*$ the linear dual space of $\gg$ and 
  by $\langle\cdot,\cdot\rangle\colon\gg^*\times\gg\to\RR$ 
  the natural duality pairing. 
  The coadjoint action of $G$ is 
$$\Ad^*_G\colon G\times\gg^*\to\gg^*, 
\quad (g,\xi)\mapsto \Ad^*_G(g)\xi=\xi\circ\Ad_G(g^{-1}). $$
Pick any $\xi_0\in\gg^*$ with its corresponding coadjoint orbit $\Oc:=\Ad_G^*(G)\xi_0\subseteq\gg^*$.  
 The isotropy group at $\xi_0$ is $G_{\xi_0}:=\{g\in G\mid\Ad_G^*(g)\xi_0=\xi_0\}$ 
 with the corresponding isotropy Lie algebra $\gg_{\xi_0}=\{X\in\gg\mid\xi_0\circ\ad_{\gg}X=0\}$.  
 
 Let $n:=\dim\gg$ and fix any sequence of ideals in $\gg$, 
$$\{0\}=\gg_0\subset\gg_1\subset\cdots\subset\gg_n=\gg$$
such that $\dim(\gg_j/\gg_{j-1})=1$ and $[\gg,\gg_j]\subseteq\gg_{j-1}$ 
for $j=1,\dots,n$. 
Pick any $X_j\in\gg_j\setminus\gg_{j-1}$ for $j=1,\dots,n$, 
so that the set $\{X_1,\dots,X_n\}$ will be a Jordan-H\"older basis in~$\gg$. 

The set of \emph{jump indices} of the coadjoint orbit $\Oc$ 
with respect to the above Jordan-H\"older basis is  
$e:=\{j\in\{1,\dots,n\}\mid \gg_j\not\subseteq\gg_{j-1}+\gg_{\xi_0}\}$ 
and does not depend on the choice of $\xi_0\in\Oc$. 
The corresponding \emph{predual of the coadjoint orbit}~$\Oc$ is  
$$\gg_e:=\spa\{X_j\mid j\in e\}\subseteq\gg$$
and it turns out that the map $\Oc\to\gg_e^*$, $\xi\mapsto\xi\vert_{\gg_e}$ is a diffeomorphism.

\begin{proposition}\label{WP}
Assume the above setting and 
let $\pi\colon G\to\BB(\Hc)$ be a fixed unitary irreducible representation 
associated with the coadjoint orbit $\Oc$. 
Then there exists a Lebesgue measure $\mu_e$ on $\gg_e$ 
for which, if we denote by $\mu$ the measure on $G$ obtained as the pusforward 
of $\mu_e$ by the map $\exp_G\vert_{\gg_e}\colon\gg_e\to G$, 
then $\pi\in\SQ(\BB(\Hc),\mu)$ and its corresponding operator 
$\Phi^\pi\colon\Hc\widehat{\otimes}\overline{\Hc}\to L^2(G,\mu)$ is unitary. 
\end{proposition}

\begin{proof}
It easily follows by \cite[Th. 2.2.6--7]{Pe} that there exists a Lebesgue measure $\mu_e$ on $\gg_e$ 
for which if we define the measure $\mu$ as in the above statement, 
then the operator $T^\pi\colon L^1(G,\mu)\cap L^2(G,\mu)\to\BB(\Hc)$ given by 
$$T^\pi f
=\int_{\gg_e}f(\exp_G x)\pi(\exp_G x) d\mu_e(x)
=\int_G f(s)\pi(s) d\mu(s)$$
extends to a unitary operator $L^2(G,\mu)\to\BB_2(\Hc)$.  
One then obtains the assertion by the same method as in the proof of Proposition~\ref{meta}. 
\end{proof}

The space of smooth vectors  
$\Hc_\infty:=\{v\in\Hc\mid \pi(\cdot)v\in\Ci(G,\Hc)\}$ 
is a Fr\'echet space in a natural way and is a dense linear subspace of $\Hc$ 
which is invariant under the unitary operator $\pi(g)$ for every $g\in G$.  
We denote by $\Hc_{-\infty}$ the space of all continuous antilinear functionals on $\Hc_\infty$ 
and 
then we have the natural inclusions $\Hc_\infty\hookrightarrow\Hc\hookrightarrow\Hc_{-\infty}$. 

Now consider the unitary representation 
$\pi \otimes {\bar \pi}\colon G\times G\to\BB(\BB_2(\Hc))$ defined by 
$$(\forall g_1,g_2\in G)(\forall T\in\BB_2(\Hc))\quad 
(\pi \otimes {\bar \pi})(g_1,g_2)T=\pi(g_1)T\pi(g_2)^{-1}.$$
It is well-known that $\pi \otimes {\bar \pi}$ is strongly continuous.  
The corresponding space of smooth vectors is denoted by $\Bc(\Hc)_\infty$ 
and is called the space of smooth operators for the representation~$\pi$. 
One can prove that actually $\Bc(\Hc)_\infty\subseteq\BB_1(\Hc)$. 

Since 
$\{\langle\cdot, f_1\rangle f_2\mid f_1,f_2\in\Hc_\infty\}
\subseteq\Bc(\Hc)_\infty\subseteq\BB_1(\Hc)$ and $\Hc_\infty$ 
is dense in $\Hc$, we obtain continuous inclusion maps 
\begin{equation*}
\Bc(\Hc)_\infty\hookrightarrow\BB_1(\Hc)\hookrightarrow\Bc(\Hc)
\hookrightarrow\Bc(\Hc)_\infty^*,
\end{equation*}
where the latter mapping is constructed by using 
the well-known isomorphism $(\BB_1(\Hc))^*\simeq\Bc(\Hc)$ 
given by the usual semifinite trace on $\Bc(\Hc)$.

We conclude by noting that the version in the present setting 
of the above dequantization formula from 
Corollary~\ref{fitofag} corresponds to \cite[Th. 2.2.6]{Pe}. 

\bigskip
\bigskip
\leftline{\textit{Acknowledgement}}
We thank the referee for numerous remarks and suggestions that helped us improve the presentation.

\end{document}